\documentclass{article}
\usepackage{jlcode}
\RequirePackage{amsmath}
\RequirePackage{amsthm}
\RequirePackage{amsfonts}
\RequirePackage{amssymb}
\RequirePackage{mathtools}  
\RequirePackage{mathrsfs}   
\RequirePackage{bbm}        
\RequirePackage{nicefrac}   
\RequirePackage{siunitx}
\RequirePackage{url}
\RequirePackage{diagbox}
\usepackage{arxiv}
\usepackage{kthcolors}
\usepackage{algorithm}
\usepackage{algpseudocode}

\algnewcommand\algorithmicinput{\hspace*{1.0em}\textbf{input:}}
\algnewcommand\Input{\item[\algorithmicinput]}
\algnewcommand\algorithmicoutput{\hspace*{1.0em}\textbf{output:}}
\algnewcommand\Output{\item[\algorithmicoutput]}
\newcommand{\algorithmicbreak}{\textbf{break}}

\newcommand{\algorithmicand}{\textbf{ and }}
\renewcommand{\And}{\algorithmicand}
\newcommand{\algorithmicor}{\textbf{ or }}

\providecommand{\given}{}
\DeclarePairedDelimiter{\parentheses}{(}{)}
\DeclarePairedDelimiter{\brackets}{[}{]}
\DeclarePairedDelimiter{\braces}{\lbrace}{\rbrace}
\DeclarePairedDelimiter{\verts}{\lvert}{\rvert}
\DeclarePairedDelimiter{\Verts}{\lVert}{\rVert}
\DeclarePairedDelimiter{\dangles}{\langle}{\rangle}

\providecommand{\SetSymbol}[1][]{%
  \nonscript\:#1\vert
  \allowbreak
  \nonscript\:
  \mathopen{}%
}
\DeclarePairedDelimiterX{\Set}[1]{\lbrace}{\rbrace}{
  \renewcommand{\given}{\SetSymbol[\delimsize]}
  #1
}

\DeclareRobustCommand{\stirling}{\genfrac\{\}{0pt}{}}

\providecommand{\st}{\ensuremath\;\mathrm{s.t.}\;}

\providecommand{\abs}[1]{\ensuremath{\verts*{#1}}}
\providecommand{\norm}[1]{\ensuremath{\Verts*{#1}}}

\DeclareMathOperator*{\minimize}{minimize}

\theoremstyle{definition}                     

\newtheorem{definition}{Definition}

\theoremstyle{remark}                         

\theoremstyle{plain}                          
\newtheorem{theorem}{Theorem}[section]
\newtheorem{corollary}[theorem]{Corollary}

\usepackage{pgfplots}
\usepackage{tikz}

\begin{document}

\title{Dynamic cut aggregation in L-shaped algorithms}

\author{Martin Biel \\
  Division of Decision and Control Systems\\
  School of EECS, KTH Royal Institute of Technology\\
  SE-100 44 Stockholm, Sweden\\
  \texttt{mbiel@kth.se} \\
  \And Mikael Johansson \\
  Division of Decision and Control Systems\\
  School of EECS, KTH Royal Institute of Technology\\
  SE-100 44 Stockholm, Sweden\\
  \texttt{mikael@kth.se}}

\date{\today}

\maketitle

\begin{abstract}
  We present a novel framework for dynamic cut aggregation in L-shaped algorithms. The aim is to improve the parallel performance of distributed L-shaped algorithms through reduced communication latency and load imbalance. We show how optimality cuts can be aggregated into arbitrary partitions without affecting convergence of the L-shaped algorithm. Furthermore, we give a worst-case bound for L-shaped algorithms with static cut aggregation and then extend this result for dynamic aggregation. We propose a variety of aggregation schemes that fit into our framework, and evaluate them on a collection of large-scale stochastic programming problems. All methods are implemented in our open-source framework for stochastic programming, \jlinl{StochasticPrograms.jl}, written in the Julia programming language. In addition, we propose a granulated strategy that combines the strengths of dynamic and static cut aggregation. Major performance improvements are possible with our approach in distributed settings. Our experimental results suggest that the granulated strategy can consistently yield high performance on a range of test problems. The experimental results are supported by our worst-case bounds.
\end{abstract}

\section{Introduction}
\label{sec:introduction}

Stochastic programming is a modeling framework for optimizing decisions under uncertainty~\cite{Birge2011}. Applications of this mathematical field include power systems~\cite{Fleten2007,Groewe-Kuska2005,petra_real-time_2014}, finance~\cite{Krokhmal2005,Zenios2005}, and transportation~\cite{Powell1987,Powell2005}. A common setting is given by linear two-stage stochastic programs. Modern large-scale industrial applications, such as the unit commitment problem presented in~\cite{petra_real-time_2014}, generally require decomposition approaches and parallelization strategies. The well-known  L-shaped algorithm~\cite{van_slyke_l-shaped_1969} is a cutting plane algorithm that can be applied to efficiently decompose and solve two-stage stochastic programs in parallel. The algorithm is essentially equivalent to Benders decomposition~\cite{benders} and can effectively exploit the structure of two-stage programs. The L-shaped algorithm decomposes a stochastic program into an approximating master problem and a set of subproblems. Each iteration the solutions of the subproblems are used to generate cutting planes that are included in the master problem to improve the approximation.

The L-shaped algorithm was originally proposed as a single-cut algorithm~\cite{van_slyke_l-shaped_1969}. In other words, all cutting planes generated during an iteration of the algorithm are aggregated into a single supporting hyperplane. L-shaped was later extended to a multi-cut variant~\cite{Birge1988}, with better convergence properties on many test examples. In contrast to single-cut, generated cutting planes are not aggregated in the multi-cut approach. In this way, more information is kept which typically improves convergence at the cost of a larger master problem. Recent contributions have explored aggregation strategies that fall between a single-cut and multi-cut approach~\cite{Trukhanov2010,Zhang2015,Yang2016}. The aim is to preserve the convergence properties of a multi-cut algorithm, while reducing the size growth of the master problem and communication overhead in distributed implementations. In this work, we formalize this approach and also present a novel dynamic aggregation procedure based on an alternative L-shaped formulation. We show how this formulation allows us to prototype various heuristic aggregation schemes in our open-source software framework for stochastic programming\footnote{Freely available on Github: \url{https://github.com/martinbiel/StochasticPrograms.jl}}~\cite{spjl}. We provide worst-case complexity results and show that large performance gains are possible in practice by solving applied problems.

The rest of this paper is organized as follows. In the next section, we give a brief introduction to the L-shaped algorithm with a focus on cut aggregation. We then present a mathematical formalization of static cut aggregation methods and provide a complexity result that compliments earlier results for single- and multi-cut methods. Next, in Section~\ref{sec:review-l-shaped}, we give a short review of prior cut aggregation approaches in terms of our introduced notation. We then introduce dynamic cut aggregation as an extension of our static framework in Section~\ref{sec:dynam-cut-aggr} and provide further theoretical results. We suggest a collection of practical dynamic aggregation schemes in Section~\ref{sec:aggregation-schemes} and implement them in our software framework. Finally, in Section~\ref{sec:numer-exper}, we conduct numerical experiments to evaluate our proposed methods before concluding the paper.

\section{Cut aggregation in L-shaped algorithms}
\label{sec:cut-aggregation-l}

We consider finite two-stage stochastic programs of the form
\begin{equation} \label{eq:sp}
  \begin{aligned}
    \minimize_{x \in \mathbb{R}^n, y_s \in \mathbb{R}^m} & \quad c^T x + \sum_{s = 1}^{N} \pi_s q_s^T y_s & \\
    \st & \quad Ax = b & \\
    & \quad T_s x + W y_s = h_s, \quad &&s = 1,\dots,N \\
    & \quad x \geq 0, \; y_s \geq 0, \quad &&s = 1,\dots,N,
  \end{aligned}
\end{equation}
where $A \in \mathbb{R}^{p \times n}$, $T_s \in \mathbb{R}^{q \times n}, \; s = 1,\dots,N$ and $W \in \mathbb{R}^{q \times m}$. Scenario-dependent data $\xi_s = \begin{pmatrix}
  q_{\omega} & T_{\omega} & h_{\omega}
\end{pmatrix}^T$ is drawn with probability $\pi_s$ from a discrete sample space $\Omega$. This formulation can be used for problems with infinite sample space and continuous distributions through sample-based techniques~\cite{saa,saacomp}. We use the natural decomposition into a first and second stage:
\begin{equation*}
  \begin{aligned}
    \minimize_{x \in \mathbb{R}^n} & \quad c^T x + \sum_{s = 1}^{N} \pi_s Q_s(x) & \\
    \st & \quad Ax = b & \\
    & \quad x \geq 0,
  \end{aligned}
\end{equation*}
where
\begin{equation*}
  \begin{aligned}
    Q_s(x) = \min_{y_s \in \mathrm{R}^m} & \quad q_s^T y_s \\
    \st & \quad Wy_s = h_s - T_s x \\
    & \quad y_s \geq 0.
  \end{aligned}
\end{equation*}

\subsection{The L-shaped algorithm}
\label{sec:l-shaped-algorithm}

The L-shaped algorithm decomposes~\eqref{eq:sp} into a master problem and $N$ subproblems. Consider the following epigraph form of~\eqref{eq:sp}:
\begin{equation} \label{eq:outerlin}
  \begin{aligned}
    \minimize_{x \in \mathbb{R}^n} & \quad c^T x + \theta \\
    \st & \quad Ax = b & \\
    & \quad \theta \geq Q(x) \\
    & \quad x \geq 0.
  \end{aligned}
\end{equation}
By duality, it holds that
\begin{equation} \label{eq:epigraph}
  \begin{aligned}
    Q(x) &= \sum_{s = 1}^{N}\pi_sQ_s(x) \\
    &= \sum_{s = 1}^{N} \pi_s\max_{\lambda_s \in \Lambda_s}{\lambda_s^T(h_s-T_sx)} \\
    &= \sum_{s = 1}^{N} \pi_s\max_{\lambda_s \in \bar{\Lambda}_s}{\lambda_s^T(h_s-T_sx)},
  \end{aligned}
\end{equation}
where $\Lambda_s = \Set{\lambda \in \mathbb{R}^q \given W^T\lambda \leq q_s}$ and $\bar{\Lambda}_s$ are the extreme points of $\Lambda_s$. Hence, a full representation of~\eqref{eq:outerlin} is given by
\begin{equation*}
  \begin{aligned}
    \minimize_{x \in \mathbb{R}^n} & \quad c^T x + \theta \\
    \st & \quad Ax = b & \\
    & \quad \theta \geq \sum_{s = 1}^{N} \pi_s{\lambda_s^T(h_s-T_sx)} && \quad \parentheses*{\lambda_1,\dots,\lambda_N} \in \bar{\Lambda}_1\times\dots\times\bar{\Lambda}_N \\
    & \quad x \geq 0.
  \end{aligned}
\end{equation*}
The main idea of the L-shaped algorithm is to generate increasingly tight supporting cutting planes of the piecewise linear function $Q(x)$. During the procedure, solution iterates $x_k$ are used to parameterize subproblems of the form:
\begin{equation} \label{eq:lssubprob}
  \begin{aligned}
    Q_s(x_k) = \min_{y_s \in \mathrm{R}^m} & \quad q_s^T y_s \\
    \st & \quad Wy_s = h_s - T_s x_k \\
    & \quad y_s \geq 0.
  \end{aligned}
\end{equation}
The optimal values are combined into the upper bound $Q(x_k) = \sum_{s = 1}^{N} Q_s(x_k)$. It follows from the same duality result that $\lambda_{s,k}^T(h_s - T_sx)$, where $\lambda_{s,k}$ is the dual optimizer of~\eqref{eq:lssubprob}, is a valid support function for $Q_s(x)$ at $x_k$. Moreover,
\begin{equation*}
  \sum_{s = 1}^{N}\pi_s \lambda_{s,k}^T(h_s-T_sx)
\end{equation*}
corresponds exactly to one of the facets of $Q(x)$ in the full representation~\eqref{eq:outerlin}. In the original formulation of the L-shaped algorithm~\cite{van_slyke_l-shaped_1969}, the above result is used in each iteration $k$ to construct \emph{optimality cuts}:
\begin{align*}
  \partial Q_{k} &= \sum_{s = 1}^{N} \pi_s\lambda_{s,k}^T T_s \\
  q_k &= \sum_{s = 1}^{N}\pi_s\lambda_{s,k}^Th_s.
\end{align*}
If the optimality cut is not satisfied by the current master iterate, i.e., if
\begin{equation*}
  \theta_k < q_k - \partial Q_{k}x_k,
\end{equation*}
then the optimality cut is included in the master problem as follows:
\begin{equation} \label{eq:lsmaster}
  \begin{aligned}
    \minimize_{x \in \mathbb{R}^n} & \quad c^T x + \theta \\
    \st & \quad Ax = b \\
    & \quad \partial Q_k x + \theta \geq q_k, \quad && \forall k \\
    & \quad x \geq 0.
  \end{aligned}
\end{equation}
If any second-stage problem is infeasible for the given $x$, \emph{feasibility cuts} can be generated and included in the master problem~\cite{van_slyke_l-shaped_1969}. The master problem is then re-solved to generate the next iterate $(x_{k+1},\theta_{k+1})$. This is repeated until the gap between the upper bound $Q(x_k)$ and lower bound $\theta_{k+1}$ becomes small, or if the latest optimality cut is already satisfied by the current master iterate, upon which the algorithm terminates. The L-shaped algorithm is finitely convergent because $W$ has a finite number of bases~\cite{van_slyke_l-shaped_1969}.

\subsection{The multi-cut L-shaped algorithm}
\label{sec:multicut-l-shaped}

The original L-shaped algorithm was extended in~\cite{Birge1988} by including all generated optimality cuts in a disaggregate form at each iteration. In other words, optimality cuts are generated for each subproblem:
\begin{align*}
  \partial Q_{s,k} &= \pi_s\lambda_{s,k}^T T_s \\
  q_{s,k} &= \pi_s\lambda_{s,k}^Th_s,
\end{align*}
and these then enter a modified master problem as follows:
\begin{equation} \label{eq:multils}
  \begin{aligned}
    \minimize_{x \in \mathbb{R}^n} & \quad c^T x + \sum_{s = 1}^{N} \theta_s \\
    \st & \quad Ax = b \\
    & \quad \partial Q_{s,k} x + \theta_s \geq q_{s,k}, \quad &&s = 1,\dots,N \quad \forall k \\
    & \quad x \geq 0.
  \end{aligned}
\end{equation}
A given disaggregate optimality cut not included in the master problem if it is already satisfied by the current master iterate, i.e., if
\begin{equation*}
  \theta_{s,k} \geq q_{s,k} - \partial Q_{s,k}x_k
\end{equation*}
The authors of~\cite{Birge1988} show that the resulting procedure will terminate in equal or fewer iterations than the original aggregate version~\cite{van_slyke_l-shaped_1969} if the major iterates coincide. A simple argument in favour of a multi-cut approach is that the master problem has more available information at each iteration and is therefore able to localize the set of optimal solutions faster. However, there is no general rule that the disaggregate master problem converges in fewer iterations for all problems. Also, the size of the master problem grows faster if the cuts are not aggregated, which has a negative effect on the time to solution. As a rule of thumb, the authors of~\cite{Birge1988} suggest that the single-cut approach should be preferred when the number of scenarios is considerably larger than the number of first stage constraints, i.e., when $N \gg p$. Finally, the authors suggest that it may be advantageous to adopt a so called ``hybrid approach'', where cuts are aggregated in separate clusters. We propose a framework around this idea, which we introduce in the following section.

\subsection{The aggregated L-shaped algorithm}
\label{sec:aggregated-l-shaped}

We develop a formalization for using aggregation in L-shaped algorithms. The hybrid approach suggested in~\cite{Birge1988} has since been explored in practice~\cite{Vladimirou1998,Trukhanov2010}, but a theoretical analysis is missing to the best of our knowledge. We devise a general framework for arbitrary aggregation approaches, including the hybrid approach. Consider the following definition.
\begin{definition}
  A \emph{partitioning scheme}
  \begin{equation} \label{eq:partitions}
    \mathcal{S} = \Set{\mathcal{S}_1, \dots, \mathcal{S}_A}
  \end{equation}
  of $N$ scenarios is a set of partitions, or aggregates, such that
  \begin{equation} \label{eq:partition-conditions}
    \begin{aligned}
      \mathcal{S}_a &\subseteq \Set{1, \dots, N}, &\qquad a = 1,\dots,A \\
      \mathcal{S}_a \cap \mathcal{S}_b &= \emptyset, &\qquad \forall a \neq b \\
      \bigcup_{a = 1}^{A} \mathcal{S}_a &= \Set{1, \dots, N}.
    \end{aligned}
  \end{equation}
\end{definition}
In an aggregated L-shaped algorithm, the results of solving subproblems in the same partition $\mathcal{S}_a$ are used to create aggregated optimality cuts
\begin{align*}
  \partial Q_{a,k} &= \sum_{s \in \mathcal{S}_a} \pi_s\lambda_{s,k}^T T_s \\
  q_{a,k} &= \sum_{s \in \mathcal{S}_a} \pi_s\lambda_{s,k}^Th_s,
\end{align*}
which then enter the master problem as follows:
\begin{equation} \label{eq:aggregatedls}
  \begin{aligned}
    \minimize_{x \in \mathbb{R}^n} & \quad c^T x + \sum_{a = 1}^{A} \theta_a \\
    \st & \quad Ax = b \\
    & \quad \partial Q_{a,k} x + \theta_a \geq q_{a,k}, \quad &&a = 1,\dots,A \quad \forall k \\
    & \quad x \geq 0.
  \end{aligned}
\end{equation}
A given aggregated optimality cut is not included in the master if it is already satisfied by the current master iterate, i.e., if
\begin{equation*}
  \theta_{a,k} \geq q_{a,k} - \partial Q_{a,k}x_k
\end{equation*}
We give a convergence proof for a general variant of this algorithm, where the partitioning scheme can vary over iterations, in a following section. Note that the partitioning scheme $\mathcal{S} = \Set{\mathcal{S}_1}$ with $\mathcal{S}_1 = \Set{1,\dots,N}$ corresponds to the original single-cut algorithm, while $\mathcal{S} = \Set{\Set{a} \given a \in \Set{1,\dots,N}}$ corresponds to the multi-cut algorithm. We introduce two entities that characterize any given partitioning scheme $\mathcal{S}$.
\begin{definition}
  The \emph{aggregation size} of the partitioning scheme $\mathcal{S}$ is given by
  \begin{equation*}
    A(\mathcal{S}) = \abs{\mathcal{S}}.
  \end{equation*}
\end{definition}
\begin{definition}
  The \emph{aggregation level} of the partitioning scheme $\mathcal{S}$ is given by
  \begin{equation*}
    A_L(\mathcal{S}) = \max_{a = 1,\dots,A(\mathcal{S})} \verts{\mathcal{S}_a}.
  \end{equation*}
\end{definition}
It is clear that $A(\mathcal{S}) = 1, \, A_L(\mathcal{S}) = N$ for single-cut L-shaped, and $A(\mathcal{S}) = N, \, A_L(\mathcal{S}) = 1$ for multi-cut L-shaped. Moreover, these values constitute the extremes in terms of these characteristics, i.e., $1 \leq A(\mathcal{S}) \leq N, \, 1 \leq A_L(\mathcal{S}) \leq N$ for any partitioning scheme $\mathcal{S}$.

We extend the worst-case complexity analysis developed in~\cite{Birge1988} to the aggregated case. Recall the following definition:
\begin{definition}
  Let $b_s$ represent the maximum number of different slopes of the piecewise linear function $Q_s(x)$ in any direction parallel to one of the axes. Then, $b = \max_s b_s$ is the \emph{slope number} of $Q(x)$.
\end{definition}
The worst-case complexity result developed by the authors of~\cite{Birge1988} is then given by the following theorem:
\begin{theorem}  \label{thm:lscomplexity}
  The maximum number of iterations required to obtain an optimal solution of~\eqref{eq:sp}, using the single-cut L-shaped algorithm, is given by
  \begin{equation} \label{eq:lsbound}
    \brackets*{1 + N(b-1)}^m,
  \end{equation}
  while the maximum number of iterations required to obtain an optimal solution of~\eqref{eq:sp}, using the multi-cut L-shaped algorithm, is given by
  \begin{equation} \label{eq:multibound}
    1 + N(b^m-1),
  \end{equation}
  where $b$ is the slope number of $Q(x)$.
\end{theorem}
Using similar arguments, we postulate and prove the following extended result for the aggregated L-shaped algorithm:
\begin{theorem} \label{thm:aggcomplexity}
  The maximum number of iterations required to obtain an optimal solution of~\eqref{eq:sp}, using an aggregated L-shaped algorithm that uses a partitioning scheme $\mathcal{S} = \Set{\mathcal{S}_1, \dots, \mathcal{S}_{A(\mathcal{S})}}$ satisfying~\eqref{eq:partition-conditions}, is given by
  \begin{equation} \label{eq:aggcomplexity}
    1 + \sum_{a = 1}^{A(\mathcal{S})} \brackets*{1+\abs{\mathcal{S}_a}(b-1)}^m - A(\mathcal{S}),
  \end{equation}
  where $b$ is the slope number of $Q(x)$.
\end{theorem}
\begin{proof}
  See~\ref{sec:proofs}.
\end{proof}
Because $\abs{\mathcal{S}_a} \leq A_L(\mathcal{S})$ holds by construction, we can bound the sum in~\eqref{eq:aggcomplexity}. This yields the following upper bound on the worst-case complexity:
\begin{corollary} \label{thm:aggsimplecomplexity}
  The maximum number of iterations of an aggregated L-shaped algorithm, using a partitioning scheme $\mathcal{S} = \Set{\mathcal{S}_1, \dots, \mathcal{S}_A}$ satisfying~\eqref{eq:partition-conditions}, is upper bounded by
  \begin{equation} \label{eq:aggbound}
    1 + A(\mathcal{S}) \parentheses*{\brackets*{1+A_L(\mathcal{S})(b-1)}^m - 1},
  \end{equation}
  where $b$ is the slope number of $Q(x)$, and $m$ is the row dimension of $W$.
\end{corollary}

Note that, the original results in Theorem~\ref{thm:lscomplexity} are recovered for the single-cut L-shaped algorithm ($A(\mathcal{S}) = 1, \, A_L(\mathcal{S}) = N$) and for the multi-cut L-shaped algorithm ($A(\mathcal{S}) = N, \, A_L(\mathcal{S}) = 1$). The upper bound~\eqref{eq:aggbound} is easier to reason about than~\eqref{eq:aggcomplexity}, but it could be pessimistic for irregular aggregation schemes where $A(\mathcal{S}) > 1$ and $A_L(\mathcal{S})$ is close to $N$. Both expressions~\eqref{eq:aggcomplexity} and~\eqref{eq:aggbound} can grow astronomically large already for medium-scale problems. However, the worst-case results still indicate which aggregation schemes could be more performant. We can observe that decreasing the aggregation level $A_L(\mathcal{S})$ decreases the worst-case complexity. In addition, we can note that the aggregated L-shaped algorithm will in general have better worst-case performance than the single-cut L-shaped algorithm for large-scale problems. For example, the worst-case complexity of a uniform partitioning scheme, where $A_L(\mathcal{S}) = N/A(\mathcal{S})$, is on the order of
\begin{equation*}
  \frac{N^m(b-1)^m}{A(\mathcal{S})^{m-1}},
\end{equation*}
as opposed to the single-cut complexity $N^m(b-1)^m$. The size of the master problem grows slower for the aggregated L-shaped algorithm than the multi-cut L-shaped algorithm because $A(\mathcal{S}) \leq N$ constraints are added at each iteration as opposed to $N$ cuts. Thus, provided that the average-case iteration complexity of an aggregated L-shaped algorithm is not far worse than the multi-cut approach, performance improvements are possible.

\section{Review of L-shaped aggregation schemes}
\label{sec:review-l-shaped}

A comprehensive review of past contributions related to algorithmic improvements of L-shaped algorithms is provided in~\cite{Rahmaniani2017} and also in the dissertation~\cite{Wolf2014a}. We give an overview of contributions related to aggregation strategies which to the best of our knowledge could be considered the state-of-the-art. We also try to identify shortcomings in these prior strategies.

\subsection{Partial cut aggregation}
\label{sec:part-cut-aggr}

The first usage of an aggregation approach of type~\eqref{eq:aggregatedls} was presented in~\cite{Vladimirou1998}. The main motivation is to reduce communication overhead in the distributed setting as well as time to solution when re-solving the master problem. The $N$ subproblems are distributed uniformly on $r$ worker nodes. This topology is then used to induce a uniform partitioning scheme $\mathcal{S} = \Set{\mathcal{S}_1, \dots, \mathcal{S}_W}$ where $\abs{\mathcal{S}_w}$ is the number of subproblems on worker $w$. This minimizes the amount of data passed from every worker at each iteration. The numerical results do not clearly favor the aggregated approach over a multi-cut approach. However, the problem sizes were only on the order of $\num{e+4}$ variables and constraints in the performed experiments.

The partial-cut approach has since been shown to be effective in various applied problems~\cite{Zhang2015,Yang2016}. Moreover, the results in~\cite{Trukhanov2010,Kwon2017} suggest that many problems are solved more efficiently with an aggregation level somewhere between the single-cut and multi-cut, i.e., partitioning schemes $\mathcal{S}$ where $1 < A_L(\mathcal{S}) < N$. However, the beneficial effect on solution time appears problem-dependent and the optimal aggregation level $A_L(\mathcal{S})$ is not known \emph{a priori}.

\subsection{Adaptive multi-cut aggregation}
\label{sec:adapt-mult-aggr}

A more recent aggregation approach is presented in~\cite{Trukhanov2010}. The authors suggest an adaptive aggregation policy, where the partitioning $\mathcal{S}^k = \Set{\mathcal{S}^k_1, \dots, \mathcal{S}^k_{A_k}}$ is allowed to vary at each iteration $k$. The master problem is of the form~\eqref{eq:aggregatedls}. Hence, if $\mathcal{S}^k \neq \mathcal{S}^{k-1}$, then the cuts generated at iteration $k$ will not form valid supports for the second-stage objective because the master variables $\braces{\theta_a}_{a = 1}^{A_k}$ adhere to a specific partitioning. This is alleviated by repartitioning the master variables to match the new partitioning $\mathcal{S}_k$. Specifically, if $\mathcal{S}_i^{k-1},\dots,\mathcal{S}_j^{k-1}$ are aggregated in $\mathcal{S}^k$, then the master variables $\theta_i,\dots,\theta_j$ are removed from~\eqref{eq:aggregatedls} and replaced by a single new variable. Cuts from previous iterations are aggregated to adhere to the new partitioning. The authors suggest that disaggregation is also possible, but intractable in practice since it requires bookkeeping of all cuts. Consequently, $A_L(\mathcal{S}^k) \geq A_L(\mathcal{S}^{k-1})$ holds for the suggested adaptive aggregation scheme. The idea is therefore to initialize with no aggregation and run the adaptive aggregation scheme with the hope of eventually identifying an efficient aggregation level for the given problem.

The authors of~\cite{Trukhanov2010} present two heuristic rules, based on a redundancy threshold and a bound on the number of aggregates, to decide how to determine the subsequent partitioning $\mathcal{S}^{k}$ based on $\mathcal{S}^{k-1}$. They mention trying other rules, based on for example cut similarity, but state that such efforts yield no significant gains in performance. The authors also perform exhaustive tests of uniform aggregation schemes of fixed size, which they refer to as \emph{static aggregation}. These results also indicate that many problems are solved faster when $1 < A_L(\mathcal{S}) < N$.

We identify a few drawbacks with adaptive aggregation. The first drawback is that the partitioning of the master variables $\theta_a,\,a = 1,\dots,A_k$ must always match the current partitioning scheme $\mathcal{S}^k$ during the adaptive procedure. Consequently, any changes to the partitioning scheme infer deleting and adding columns in the master problem, which can lead to significant overhead for large sparse problems. Moreover, the cuts from previous iterations have to be updated to adhere to the new partitioning. This incurs a large number of constraint replacements, which also increases the overhead in master iterations. The second drawback is that the nature of the implementation makes disaggregation of cuts non-performant. Therefore, the partitioning can only be made coarser. Finally, even though the adaptive method is introduced to overcome the fact that the optimal aggregation level is not known a priori, the method is governed by tunable parameters whose values are shown to greatly influence runtime.

\subsection{Cut consolidation}
\label{sec:cut-consolidation}

Another aggregation technique is presented in~\cite{Wolf2013}. The technique, \emph{cut consolidation}, is adopted to reduce the size of the master problem, and acts independently of the aggregation scheme used. The idea is to prune historical cuts that have become inactive, but retain their aggregation to keep some information in the master. Specifically, the following consolidation scheme is used. If the number of cuts
\begin{equation*}
  \partial Q_{s,k} x + \theta_s \geq q_{s,k}, \quad s = 1,\dots,N
\end{equation*}
from a previous iteration $k$ that are inactive in the master reaches a user-defined threshold, then all cuts from iteration $k$ are removed from the master and the special aggregate
\begin{equation*}
  \sum_{s = 1}^{N}\partial Q_{s,k} x + \sum_{s = 1}^{N}\theta_s \geq \sum_{s = 1}^{N} q_{s,k}
\end{equation*}
is added instead. Numerical results indicate that cut consolidation can considerably reduce the time to solution, especially in combination with a partial aggregation scheme. Similar to the adaptive aggregation approach, the proposed cut consolidation is governed by two tunable threshold parameters which have non-negligent impact on the runtime. Furthermore, cut consolidation does not reduce communication latency from cut passing. This is however not a large issue as the method can be combined naturally with partial cut aggregation.

\section{Dynamic cut aggregation}
\label{sec:dynam-cut-aggr}

We propose a new aggregation procedure, which we call \emph{dynamic cut aggregation}. We introduce the procedure and derive convergence and complexity results in this section and then propose practical implementations in the next section. The main idea of our approach is to retain the structure of the multi-cut master problem~\eqref{eq:multils}, while still allowing for a dynamic partitioning scheme that can vary over iterations. We build upon the concepts introduced in Section~\ref{sec:aggregated-l-shaped}.
\begin{definition}
  A \emph{dynamic partitioning scheme}
  \begin{equation} \label{eq:dynscheme}
    \mathcal{D} = \braces{\mathcal{S}^k}_{k = 1}^{\infty}
  \end{equation}
  is a sequence of partitioning schemes $\mathcal{S}^k = \Set{\mathcal{S}^k_1, \dots, \mathcal{S}^k_{A_k}},$ each satisfying~\eqref{eq:partition-conditions}.
\end{definition}
Next, we pose an L-shaped algorithm with dynamic cut aggregation. Our reformulated master problem has the following form:
\begin{equation} \label{eq:dynamicls}
  \begin{aligned}
    \minimize_{x \in \mathbb{R}^n} & \quad c^T x + \sum_{s = 1}^{N} \theta_s \\
    \st & \quad Ax = b  \\
    & \quad \sum_{s \in \mathcal{S}^{k}_{a}} \partial Q_{s,k} x + \sum_{s \in \mathcal{S}^{k}_{a}} \theta_s \geq \sum_{s \in \mathcal{S}^{k}_{a}} q_{s,k}, \quad && a = 1,\dots,A_k,\quad \mathcal{S}^k \in \mathcal{D} \quad \forall k \\
    & \quad x \geq 0.
  \end{aligned}
\end{equation}
Again, a new cut aggregate is only added to the master problem if it is not satisfied by the current master iterate, i.e., if
\begin{equation*}
  \sum_{s \in \mathcal{S}^k_a} \theta_{s,k} < \sum_{s \in \mathcal{S}^k_a} \parentheses*{q_{s,k} - \partial Q_{s,k} x_k}. \quad \mathcal{S}^k_a \in \mathcal{S}^k
\end{equation*}
If the partitioning scheme is fixed every iteration $\mathcal{S}^k = \Set{\mathcal{S}_1,\dots,\mathcal{S}_A}$, the aggregated master problem~\eqref{eq:aggregatedls} is recovered through the variable substitutions
\begin{equation*}
  \theta_a = \sum_{s \in \mathcal{S}_a} \theta_s \quad a = 1,\dots,A.
\end{equation*}

With our reformulation~\eqref{eq:dynamicls}, the number of master columns is not affected by the changes to the partitioning scheme and cuts from previous iterations remain valid. Moreover, disaggregation is possible; so, the partitioning scheme can vary between single-cut and multi-cut at each iteration. In this way, we address some drawbacks of the adaptive aggregation method. However, our approach requires separate second-stage objective variables $\theta_s$ in the master problem for each of the $N$ subproblems. This increases the memory footprint of the master problem as the scenario count $N$ grows, which reduces the scalability of the approach. Moreover, without specifying how the partitioning schemes $\mathcal{S}^k$ should be chosen the method is not practical. However, the flexibility of the formulation~\eqref{eq:dynamicls} allows us to formulate a large variety of implementable aggregation schemes, which we present Section~\ref{sec:aggregation-schemes}. We will also propose a scheme that can overcome the drawback of an increased number of master columns.

\subsection{Convergence}
\label{sec:convergence}

We give a proof of finite convergence for the L-shaped algorithm with dynamic cut aggregation.
\begin{theorem} \label{thm:dynlshapedconvergence}
  An L-shaped algorithm that uses dynamic cut aggregation, with a dynamic partitioning scheme $\mathcal{D} = \braces{\mathcal{S}^k}_{k = 1}^{\infty}$ for which the partitioning scheme $\mathcal{S}^k$ at each iteration satisfies the conditions~\eqref{eq:partition-conditions}, converges to an optimal solution of~\eqref{eq:sp} in a finite number of iterations.
\end{theorem}
\begin{proof}
See~\ref{sec:proofs}.
\end{proof}

\subsection{Complexity}
\label{sec:complexity}

Since the partitioning scheme used in dynamic cut aggregation can vary with iterations, the worst-case result in Theorem~\ref{thm:aggcomplexity} does not hold and must be extended. First, we introduce some well-known combinatorial concepts that are required in the analysis.
\begin{definition}
  A \emph{k-combination} of $N$ elements is a subset of $1,\dots,N$ of size $k$. The number of $k$-combinations out of $N$ elements is denoted by $\binom{N}{k}$.
\end{definition}
\begin{definition}
  The \emph{Stirling number of the second kind} is the number of ways to partition $N$ elements into $k$ non-empty subsets, and is denoted by $\stirling{N}{k}$.
\end{definition}
\begin{definition}
  The $N$th \emph{Bell number}, denoted by $B_N$, is the number of possible partitionings of $N$ elements. In terms of Stirling numbers it is given by
  \begin{equation*}
    B_N = \sum_{k = 1}^{N}\stirling{N}{k}
  \end{equation*}
\end{definition}
We can now postulate and prove the following result for dynamic cut aggregation:
\begin{theorem} \label{thm:dynaggcomplexity}
  The maximum number of iterations required to obtain an optimal solution of~\eqref{eq:sp}, using an L-shaped algorithm that uses dynamic cut aggregation with a dynamic partitioning scheme $\mathcal{D} = \braces{\mathcal{S}^k}_{k = 1}^{\infty}$, is given by
  \begin{equation} \label{eq:dynaggcomplexity}
    2 + \sum_{a_L = 1}^{N} \binom{N}{a_L} \brackets*{1+a_L(b-1)}^m - \sum_{a_L = 1}^{N} \stirling{N}{a_L} - A_0,
  \end{equation}
  where $b$ is the slope number of $Q(x)$.
\end{theorem}
\begin{proof}
See~\ref{sec:proofs}.
\end{proof}
We can obtain a tighter bound by imposing restrictions on the dynamic partitioning scheme. For example, we can limit the size of the aggregates at each iteration, which simply removes summands in~\eqref{eq:dynaggcomplexity}. The following result is obtained:
\begin{corollary} \label{cor:boundeddynaggcomplexity}
  The maximum number of iterations of an L-shaped algorithm with dynamic cut aggregation, where the dynamic partitioning scheme $\mathcal{D}$ satisfies
  \begin{equation*}
      \underline{A}_L(\mathcal{D}) \leq A_L(\mathcal{S}^k) \leq \bar{A}_L(\mathcal{D}) \quad \forall \mathcal{S}^k \in \mathcal{D}
  \end{equation*}
  is given by
  \begin{equation} \label{eq:boundeddynaggcomplexity}
    2 + \sum_{a_L = \underline{A}_L(\mathcal{D})}^{\bar{A}_L(\mathcal{D})} \binom{N}{a_L} \brackets*{1+a_L(b-1)}^m - \sum_{a_L = \underline{A}_L(\mathcal{D})}^{\bar{A}_L(\mathcal{D})} \stirling{N}{a_L} - A_0,
  \end{equation}
  where $b$ is the slope number of $Q(x)$.
\end{corollary}
We can again recover the original worst-case results presented in~\cite{Birge1988}. The single-cut L-shaped algorithm corresponds to a dynamic aggregation scheme with $\underline{A}_L = \bar{A}_L = N$ and $A_0 = 1$, for which we obtain
\begin{equation*}
  2 + \binom{N}{N} \brackets*{1 + N(b-1)}^m - \stirling{N}{N} - 1 = \brackets*{1 + N(b-1)}^m
\end{equation*}
Likewise, the multi-cut L-shaped algorithm corresponds to a dynamic aggregation scheme with $\underline{A}_L = \bar{A}_L = 1$ and $A_0 = N$, for which we obtain
\begin{equation*}
  2 + \binom{N}{1} \brackets*{1+b-1}^m - \stirling{N}{1} - N = 1 + N(b^m-1).
\end{equation*}
As with static aggregation, we can improve the worst-case bound by decreasing the aggregation level of the partitioning schemes. In addition, we would expect performance improvements from any dynamic aggregation rule that limits the possible aggregate combinations.

\subsection{Practical complexity}
\label{sec:practical-complexity}

The practical performance of dynamic aggregation schemes could be much better than suggested by the worst-case bound~\eqref{eq:dynaggcomplexity}. Any form of cut aggregation generally improves scalability in a distributed setting. Both communication latency and load imbalance among the master node and worker nodes are reduced. This holds since fewer cuts are passed from workers and the master problem does not grow as fast. Therefore, if the average iteration complexity of an aggregated L-shaped algorithm is comparable to the average multi-cut complexity in the single-core setting, then wall-clock time to solution can be greatly reduced if the aggregated L-shaped is run in parallel on distributed memory. It is not a general rule as the aggregation overhead could outweigh the gains from aggregation.

\section{Dynamic aggregation schemes}
\label{sec:aggregation-schemes}

The theoretical results derived in the previous section hold for any dynamic partitioning scheme. However, a practical implementation requires rules for how the partitioning scheme $\mathcal{S}^k$ should be chosen each iteration $k$. In this section, we propose a variety of dynamic aggregation schemes that are viable to implement and utilize in an L-shaped algorithm. The schemes that we suggest are all heuristic. Advised by our theoretical results, the schemes are designed in way that allows us to control the aggregation size, the aggregation level, as well as the possible aggregate combinations. Our software framework \jlinl{StochasticPrograms.jl} contains a large collection of documented\footnote{\url{https://martinbiel.github.io/StochasticPrograms.jl/dev/}} aggregation schemes that are based on this design philosophy. Here, we only introduce the schemes that are included in the numerical experiments presented in the paper. Note, that for comparison we have also implemented partial cut aggregation in its original formulation~\eqref{eq:aggregatedls}.

\subsection{Dynamic aggregation}
\label{sec:dynamic-aggregation}

\sloppy
The first proposed aggregation scheme is \emph{dynamic aggregation}. This scheme uses a fixed-length partitioning $\mathcal{S}^k = \Set{\mathcal{S}_1^k,\dots,\mathcal{S}_A^k}$ where each aggregate $\mathcal{S}_a^k$ can vary over iterations. A new optimality cut is placed in one of the aggregates based on a predefined selection rule. If the selection rule determines the chosen aggregate to be full, then the aggregate is added to the master problem and is then emptied. After all scenarios have been considered, any remaining non-empty aggregate is added to the master problem.

\sloppy
A selection rule returns an aggregate index $1 \leq a \leq A$ based on the aggregates $\mathcal{S}_1,\dots,\mathcal{S}_A$ and the cut candidate. The rule also determines if the chosen aggregate $\mathcal{S}_a$ should be considered full and added to the master problem. We have implemented a collection of selection rules in \jlinl{StochasticPrograms.jl}. Below, we list the subset of selection rules that we have included in the numerical experiments.

\textbf{SelectUniform}: Selects aggregates so that $|\mathcal{S}_a| = T, a = 1,\dots,A$ for some predefined $T$, with $TA \geq N$. This rule replicates partial cut aggregation, using formulation~\eqref{eq:dynamicls} instead of formulation~\eqref{eq:aggregatedls}. If $N$ is not divisible by $T$, then the final aggregate in the partition will consist of fewer than $T$ cuts. The worst-case bound for static aggregation~\eqref{eq:aggbound} is recovered for this rule.

\textbf{SelectClosest}: Selects the aggregate that is currently closest to the considered cut. Closeness is measured by a predefined distance function. We propose a set of distance functions in~\ref{sec:distance-measures}. If all aggregates are empty or no aggregate is close to the cut candidate within some relative tolerance $\tau$, then the cut candidate is placed in the next available empty aggregate. The aggregation level will depend on the chosen distance tolerance and chosen distance measure. In general, the number of possible aggregate combinations can be decreased by lowering $\tau$.

\subsection{Cluster aggregation}
\label{sec:cluster-aggregation}

\sloppy
The second proposed aggregation scheme is \emph{cluster aggregation}. The idea is to keep all new cuts in a buffer each iteration and aggregate only when all information is available. In this way, it could be possible to determine a more effective aggregation, albeit at the cost of larger overhead. A predefined cluster rule sorts the buffered cuts into a set of partitions $\mathcal{S}^k = \Set{\mathcal{S}_1^k,\dots,\mathcal{S}_{A_k}^k}$. As for dynamic aggregation, \jlinl{StochasticPrograms.jl} includes a collection of clustering rules, but we only introduce the rule used in the numerical experiments.

\sloppy
\textbf{K-medoids}: Sorts the cuts using k-medoids clustering~\cite{kmedoids}. K-medoids is an extension of the k-means algorithm for generalized distances. We cannot put precise bounds on the aggregation level because this will depend on the results of the k-medoid algorithm. Indirectly, the resulting clusters depend on the distance measure used. We expect that increasing $k$ will both decrease the aggregation level and reduce the possible aggregate combinations.

\subsection{Granulated aggregation}
\label{sec:gran-aggr-}

Finally, we propose the scheme \emph{granulated aggregation} with the aim to improve the scalability of the dynamic approaches. An apparent drawback with our aggregation formulation is that all $N$ master variables are kept disaggregate independent of the aggregation scheme. Consequently, any performance improvements from dynamic aggregation might be lost as $N$ increases due to numerical instability and the memory requirement of the master columns. To alleviate this, we combine static and dynamic aggregation. The idea is to fix an initial static aggregation scheme $\mathcal{S}$ and fix the master variables to $\theta_a, \, a = 1,\dots,A(\mathcal{S})$ according to this scheme throughout the procedure. In other words, we apply a dynamic aggregation procedure, but with fewer master variables. In this way, we can employ our novel aggregation schemes on large-scale problems without scalability issues. Furthermore, the initial partitioning naturally limits the number of possible aggregate combinations which improves the worst-case bound. An immediate drawback is that efficient partitionings that could have otherwise been identified by a dynamic scheme is missed due to the initial partitioning. However, our static worst-case bound~\eqref{eq:aggbound}, as well as our experimental results presented in the next section, indicate that this may be negligible as the L-shaped procedure progresses. The initial partitioning is arbitrary and any of the dynamic or clustering based schemes we have devised can be used in this granulated approach. For simplicity, we use a uniform scheme for the initial partitioning and leave testing other possibilities as future work.

\section{Numerical experiments}
\label{sec:numer-exper}

We benchmark the various aggregation schemes on a collection of applied problems. Our software framework supports reading stochastic program descriptions in the SMPS format~\cite{smps}. Moreover, we can sample stochastic programming instances of arbitrary scenario size using the loaded description. We use a testset which was presented and made openly available\footnote{\url{http://pages.cs.wisc.edu/~swright/stochastic/sampling/}} by the authors of~\cite{saacomp}. The authors use a sample average approximation algorithm to determine tight confidence intervals around the optimal values of the test problems. Reasonably tight confidence intervals are obtained for all test problems using a sample size of $\num{5000}$. We use this as a baseline when running the experiments. The problems \texttt{LandS} and \texttt{gbd} have relatively small first and second stage problems. As a result, they are solved fast by most methods. In order to include these two problems in the experiments and still obtain useful measurements we use sampled instances of $\num{100000}$ scenarios. This also allows us to explore the hypothesized effects on dynamic aggregation performance from using a large number of scenarios. In addition to the testset, we also consider an energy bidding problem that we have studied previously~\cite{dayahead}. The so called \texttt{dayahead} problem is formulated to to determine optimal order strategies on the Nordic day-ahead market from the perspective of a price-taking hydropower producer. The problem dimensions of \texttt{dayahead} are too large to sample $\num{5000}$ scenarios in our hardware setup. We known from experience that a reasonably tight confidence interval is obtained by sampling $\num{1000}$ scenarios, which is what we use in the experiments. We provide a summary of the testset problems and their respective problem dimensions in Table~\ref{tab:test_problems}.

\begin{table}[htbp]
  \centering
  \begin{tabular}{|clccc|}
    \hline
    \textbf{Name} & \textbf{Application} & \textbf{First-stage size}& \textbf{Second-stage size} & \textbf{Sample size}\\
    \hline
    \texttt{LandS} & Electricity planning & $(2,4)$ & $(7,12)$ & $\num{100000}$\\
    \texttt{gbd} & Aircraft allocation & $(4,17)$ & $(5,10)$ & $\num{100000}$ \\
    \texttt{20term} & Vehicle assignment & $(3,64)$ & $(124,764)$ & $\num{5000}$ \\
    \texttt{ssn} & Telecom network design & $(1,89)$ & $(175,706)$ & $\num{5000}$ \\
    \texttt{storm} & Cargo flight scheduling & $(185,121)$ & $(528,1259)$ & $\num{5000}$ \\
    \texttt{dayahead} & Energy market bidding & $(1457,1433)$ & $(2509,1909)$ & $\num{1000}$ \\
    \hline
  \end{tabular}
  \caption{Testset description.}
  \label{tab:test_problems}
\end{table}

The experiments are performed in a multi-node setup. The master node is a laptop computer with a $\num{2.6}$ GHz Intel Core i7 processor and $\num{16}$ GB of RAM. We spawn workers on a remote multi-core machine with two $\num{3.1}$ GHz Intel Xeon processors (total $\num{32}$ cores) and $\num{128}$ GB of RAM. The two machines were $\num{30}$ kilometers apart at the time of the experiments so communication latency is not negligible. Throughout, the Gurobi optimizer~\cite{gurobi} is used to solve emerging subproblems. In every experiment, the problem instance is solved to a relative tolerance of $\num{e-2}$.

\subsection{Empirical complexity}
\label{sec:empirical-complexity}

We introduce three entities for measuring empirical performance of an L-shaped algorithm. Consider the following definitions.
\begin{definition}
  The \emph{empirical iteration complexity}, denoted by $N^{\tau}_I$, is the number of iterations required for a given L-shaped algorithm to converge to an optimal solution of a given problem, within some relative tolerance $\tau$.
\end{definition}
\begin{definition}
  The \emph{empirical cut complexity}, denoted by $N^{\tau}_C$, is the number of optimality cuts in the master problem of after a given L-shaped algorithm has converged to the optimal solution of a given problem, within some relative tolerance $\tau$.
\end{definition}
\begin{definition}
  The \emph{empirical time complexity}, denoted by $N^{\tau}_T$, is the wall-clock time required for a given L-shaped algorithm to convergence to an optimal solution of a given problem, within some relative tolerance $\tau$.
\end{definition}
If an aggregated L-shaped algorithm has comparable empirical iteration complexity with that of the multi-cut L-shaped algorithm, but smaller empirical cut complexity, it is expected to perform better in a distributed setting. For most problems, we would expect a trade-off between these quantities. The worst-case bounds indicate that coarse aggregation schemes, with fewer cuts, require more iterations to converge. Likewise, fine aggregation schemes yield more cuts but fewer iterations to converge. The empirical complexities will not map directly to wall-clock time to solution, but we will show that low empirical cut complexity is a good indicator for when aggregation can yield better performance. These entities will be used to present our numerical results and more easily reason about them.

\subsection{Small-scale experiments}
\label{sec:small-scale-exper}

Most of the proposed aggregation schemes have a set of tunable parameters, and the optimal parameter values are not known for a given problem instance. Therefore, we first conduct a small experiment to explore the influence of the parameters on the performance of the algorithm. To that end, we consider the \texttt{20term} problem with $N = \num{1000}$ sampled scenarios. In the interest of time, the small-scale experiment is conducted using all $\num{32}$ worker cores. We run experiments for all suggested aggregation schemes, varying their respective parameters between the extremes. We measure all empirical complexities introduced in the previous section and report their change as aggregation parameters are varied. To make comparisons easier, we rescale the results and report relative complexities. The empirical cut complexity is presented relative to the cut complexity of multi-cut L-shaped, which is expected to yield the largest number of cuts. In contrast, empirical iteration- and time complexity are presented relative to full cut aggregation, which is expected to require the most iterations to converge.

First, we vary the aggregation level $T$ of a classical partial cut aggregation method between $1$ and $N/32$, which is the maximum aggregation possible on $\num{32}$ workers. The results are shown in Fig.~\ref{fig:20term_partial}. There is an apparent trade-off between iteration complexity and cut complexity as $T$ is varied, where the end points effectively yield multi-cut and single-cut L-shaped. This is supported by the worst-case bound~\eqref{eq:aggbound} as iteration complexity is expected to increase with coarser aggregation. Next, we perform the same test but for a dynamic aggregation schemed under the \textbf{SelectUniform} rule. By construction, this method should produce the same iterates as the standard partial approach. However, the results shown in Fig.~\ref{fig:20term_uniform} indicate that this does not hold. When we observe the procedures in detail we notice that the first few iterates are identical, but it appears that round-off errors in the solutions cause them to eventually diverge. The overall behaviour is however similar, and neither of the schemes is the most performant for all parameter choices. The results agree with prior works in that the best performance is achieved for an aggregation level between multi-cut and single-cut. In addition, the optimal aggregation level $T$ is hard to guess a priori.

\begin{figure}
  \centering
  \resizebox{0.9\textwidth}{!}{
  \begin{tikzpicture}[]
    \begin{axis}[height = {53.269444444444446mm}, legend pos = {north west}, ylabel = {Relative complexity}, title = {\texttt{20term} -Partial cut aggregation}, xmin = {1}, xmax = {32}, unbounded coords=jump,scaled x ticks = false, xlabel = {$T$}, xlabel style = {font = {\fontsize{14 pt}{14.3 pt}\selectfont}, rotate = 0.0},xmajorgrids = true, xtick align = inside,xticklabel style = {font = {\fontsize{12 pt}{10.4 pt}\selectfont}, rotate = 0.0},x grid style = {color = kth-lightgray,
        line width = 0.25,
        solid},axis x line* = left,x axis line style = {line width = 1,
        solid},scaled y ticks = false,ylabel style = {font = {\fontsize{14 pt}{14.3 pt}\selectfont}, rotate = 0.0},ymajorgrids = true,ytick align = inside,yticklabel style = {font = {\fontsize{12 pt}{10.4 pt}\selectfont}, rotate = 0.0},y grid style = {color = kth-lightgray,
        line width = 0.25,
        solid},axis y line* = left,y axis line style = {line width = 1,
        solid},    xshift = 0.0mm,
      yshift = 0mm,
      ,title style = {font = {\fontsize{14 pt}{18.2 pt}\selectfont}, rotate = 0.0},legend style = {line width = 1, xshift = 14mm,
        solid,font = {\fontsize{10 pt}{10.4 pt}\selectfont}},colorbar style={title=}, ymin = {0}, width = {152.4mm}]\addplot+[draw=none, color = kth-blue,
      line width = 0,
      solid,mark = square*,
      mark size = 2.0,
      mark options = {
        color = black,
        fill = kth-blue,
        line width = 1,
        rotate = 0,
        solid
      }] coordinates {
        (1.0, 1.0)
        (3.0, 0.6616541353383458)
        (4.0, 0.6332631578947369)
        (6.0, 0.5081503759398496)
        (7.0, 0.48)
        (9.0, 0.5139248120300752)
        (10.0, 0.36186466165413533)
        (12.0, 0.3630676691729323)
        (14.0, 0.3868872180451128)
        (15.0, 0.40926315789473683)
        (17.0, 0.3363609022556391)
        (18.0, 0.29786466165413533)
        (20.0, 0.22039097744360903)
        (22.0, 0.2911278195488722)
        (23.0, 0.3195187969924812)
        (25.0, 0.26177443609022555)
        (26.0, 0.20018045112781954)
        (28.0, 0.315187969924812)
        (29.0, 0.36138345864661653)
        (32.0, 0.3238496240601504)
      };
      \addlegendentry{$N^{0.01}_C$}
      \addplot+ [color = kth-blue,
      line width = 1,
      solid,mark = none,
      mark size = 2.0,
      mark options = {
        color = black,
        fill = kth-blue,
        line width = 1,
        rotate = 0,
        solid
      },forget plot] coordinates {
        (1.0, 1.0)
        (3.0, 0.6616541353383458)
        (4.0, 0.6332631578947369)
        (6.0, 0.5081503759398496)
        (7.0, 0.48)
        (9.0, 0.5139248120300752)
        (10.0, 0.36186466165413533)
        (12.0, 0.3630676691729323)
        (14.0, 0.3868872180451128)
        (15.0, 0.40926315789473683)
        (17.0, 0.3363609022556391)
        (18.0, 0.29786466165413533)
        (20.0, 0.22039097744360903)
        (22.0, 0.2911278195488722)
        (23.0, 0.3195187969924812)
        (25.0, 0.26177443609022555)
        (26.0, 0.20018045112781954)
        (28.0, 0.315187969924812)
        (29.0, 0.36138345864661653)
        (32.0, 0.3238496240601504)
      };
      \addplot+[draw=none, color = kth-red,
      line width = 0,
      solid,mark = triangle*,
      mark size = 3.0,
      mark options = {
        color = black,
        fill = kth-red,
        line width = 1,
        rotate = 0,
        solid
      }] coordinates {
        (1.0, 0.10007412898443291)
        (3.0, 0.18680504077094143)
        (4.0, 0.2453669384729429)
        (6.0, 0.26241660489251295)
        (7.0, 0.2972572275759822)
        (9.0, 0.3973313565604151)
        (10.0, 0.28020756115641215)
        (12.0, 0.374351371386212)
        (14.0, 0.3988139362490734)
        (15.0, 0.4217939214232765)
        (17.0, 0.519644180874722)
        (18.0, 0.4603409933283914)
        (20.0, 0.34099332839140106)
        (22.0, 0.44996293550778355)
        (23.0, 0.4936990363232024)
        (25.0, 0.4047442550037065)
        (26.0, 0.30985915492957744)
        (28.0, 0.4870274277242402)
        (29.0, 0.5581912527798369)
        (32.0, 1.0)
      };
      \addlegendentry{$N^{0.01}_I$}
      \addplot+ [color = kth-red,
      line width = 1,
      solid,mark = none,
      mark size = 2.0,
      mark options = {
        color = black,
        fill = kth-red,
        line width = 1,
        rotate = 0,
        solid
      },forget plot] coordinates {
        (1.0, 0.10007412898443291)
        (3.0, 0.18680504077094143)
        (4.0, 0.2453669384729429)
        (6.0, 0.26241660489251295)
        (7.0, 0.2972572275759822)
        (9.0, 0.3973313565604151)
        (10.0, 0.28020756115641215)
        (12.0, 0.374351371386212)
        (14.0, 0.3988139362490734)
        (15.0, 0.4217939214232765)
        (17.0, 0.519644180874722)
        (18.0, 0.4603409933283914)
        (20.0, 0.34099332839140106)
        (22.0, 0.44996293550778355)
        (23.0, 0.4936990363232024)
        (25.0, 0.4047442550037065)
        (26.0, 0.30985915492957744)
        (28.0, 0.4870274277242402)
        (29.0, 0.5581912527798369)
        (32.0, 1.0)
      };
      \addplot+[draw=none, color = kth-green,
      line width = 0,
      solid,mark = *,
      mark size = 2.0,
      mark options = {
        color = black,
        fill = kth-green,
        line width = 1,
        rotate = 0,
        solid
      }] coordinates {
        (1.0, 0.44435591204145386)
        (3.0, 0.3583042746546458)
        (4.0, 0.3990451283916541)
        (6.0, 0.37182292382949683)
        (7.0, 0.39252292057290605)
        (9.0, 0.5029081028438036)
        (10.0, 0.32689386088615174)
        (12.0, 0.4576888104664997)
        (14.0, 0.4737832251825845)
        (15.0, 0.49786998054324766)
        (17.0, 0.5861459256637835)
        (18.0, 0.4967728129227612)
        (20.0, 0.31623402221948604)
        (22.0, 0.45433191948053436)
        (23.0, 0.5920656564954748)
        (25.0, 0.48312751804541487)
        (26.0, 0.2960487931366824)
        (28.0, 0.5609493166188755)
        (29.0, 0.7055031362661132)
        (32.0, 1.0)
      };
      \addlegendentry{$N^{0.01}_T$}
      \addplot+ [color = kth-green,
      line width = 1,
      solid,mark = none,
      mark size = 2.0,
      mark options = {
        color = black,
        fill = kth-green,
        line width = 1,
        rotate = 0,
        solid
      },forget plot] coordinates {
        (1.0, 0.44435591204145386)
        (3.0, 0.3583042746546458)
        (4.0, 0.3990451283916541)
        (6.0, 0.37182292382949683)
        (7.0, 0.39252292057290605)
        (9.0, 0.5029081028438036)
        (10.0, 0.32689386088615174)
        (12.0, 0.4576888104664997)
        (14.0, 0.4737832251825845)
        (15.0, 0.49786998054324766)
        (17.0, 0.5861459256637835)
        (18.0, 0.4967728129227612)
        (20.0, 0.31623402221948604)
        (22.0, 0.45433191948053436)
        (23.0, 0.5920656564954748)
        (25.0, 0.48312751804541487)
        (26.0, 0.2960487931366824)
        (28.0, 0.5609493166188755)
        (29.0, 0.7055031362661132)
        (32.0, 1.0)
      };
    \end{axis}
  \end{tikzpicture}
}
  \caption{Relative empirical complexities when solving \texttt{20term} with $\num{1000}$ scenarios using a uniform partial aggregation scheme, as a function of the aggregation level $T$.}
  \label{fig:20term_partial}
\end{figure}
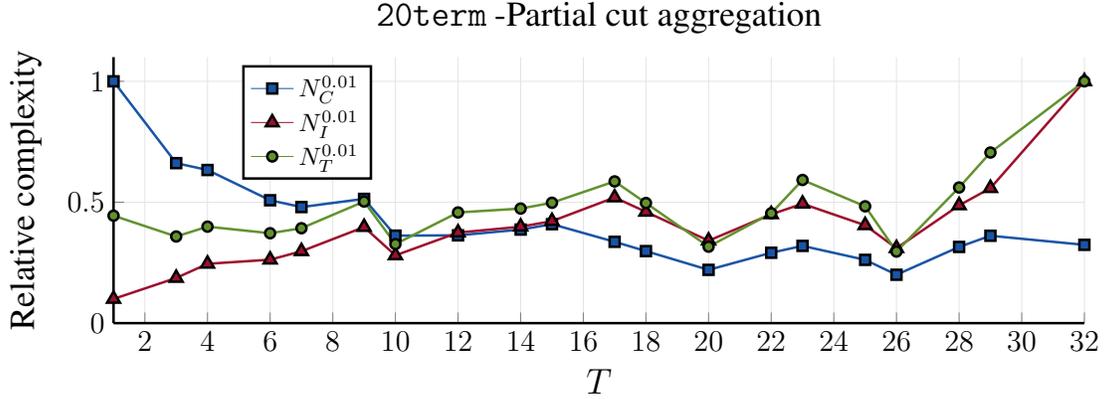

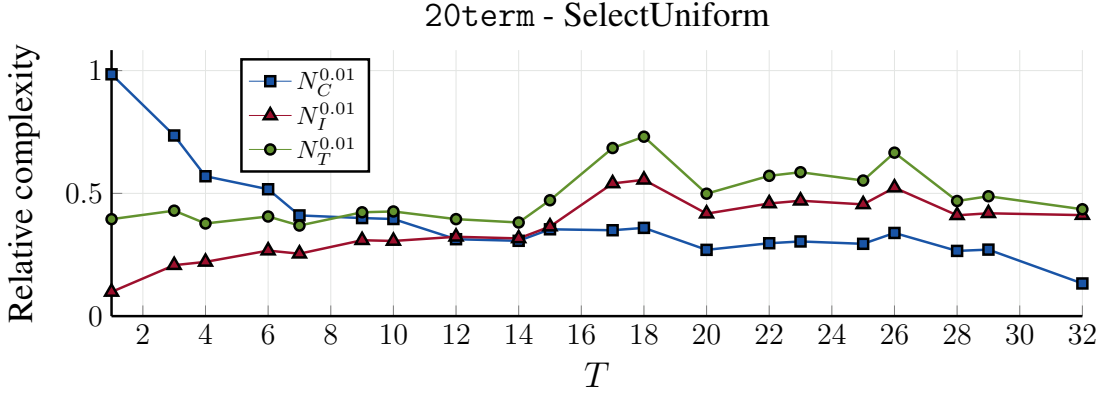
\begin{figure}
  \centering
  \resizebox{0.9\textwidth}{!}{
  \begin{tikzpicture}[]
    \begin{axis}[height = {53.269444444444446mm}, legend pos = {north west}, ylabel = {Relative complexity}, title = {\texttt{20term} - SelectUniform}, xmin = {1}, xmax = {32}, unbounded coords=jump,scaled x ticks = false, xlabel = {$T$}, xlabel style = {font = {\fontsize{14 pt}{14.3 pt}\selectfont}, rotate = 0.0},xmajorgrids = true, xtick align = inside,xticklabel style = {font = {\fontsize{12 pt}{10.4 pt}\selectfont}, rotate = 0.0},x grid style = {color = kth-lightgray,
        line width = 0.25,
        solid},axis x line* = left,x axis line style = {line width = 1,
        solid},scaled y ticks = false,ylabel style = {font = {\fontsize{14 pt}{14.3 pt}\selectfont}, rotate = 0.0},ymajorgrids = true,ytick align = inside,yticklabel style = {font = {\fontsize{12 pt}{10.4 pt}\selectfont}, rotate = 0.0},y grid style = {color = kth-lightgray,
        line width = 0.25,
        solid},axis y line* = left,y axis line style = {line width = 1,
        solid},    xshift = 0.0mm,
      yshift = 0mm,
      ,title style = {font = {\fontsize{14 pt}{18.2 pt}\selectfont}, rotate = 0.0},legend style = {line width = 1, xshift = 14mm,
        solid,font = {\fontsize{10 pt}{10.4 pt}\selectfont}},colorbar style={title=}, ymin = {0}, width = {152.4mm}]\addplot+[draw=none, color = kth-blue,
      line width = 0,
      solid,mark = square*,
      mark size = 2.0,
      mark options = {
        color = black,
        fill = kth-blue,
        line width = 1,
        rotate = 0,
        solid
      }] coordinates {
        (1.0, 0.9849624060150376)
        (3.0, 0.7357593984962406)
        (4.0, 0.5697443609022557)
        (6.0, 0.516812030075188)
        (7.0, 0.41022556390977444)
        (9.0, 0.39939849624060153)
        (10.0, 0.39554887218045115)
        (12.0, 0.31326315789473685)
        (14.0, 0.3067669172932331)
        (15.0, 0.35368421052631577)
        (17.0, 0.3498345864661654)
        (18.0, 0.35945864661654137)
        (20.0, 0.2699548872180451)
        (22.0, 0.2969022556390977)
        (23.0, 0.3041203007518797)
        (25.0, 0.29449624060150376)
        (26.0, 0.3387669172932331)
        (28.0, 0.26562406015037593)
        (29.0, 0.2709172932330827)
        (32.0, 0.13305263157894737)
      };
      \addlegendentry{$N^{0.01}_C$}
      \addplot+ [color = kth-blue,
      line width = 1,
      solid,mark = none,
      mark size = 2.0,
      mark options = {
        color = black,
        fill = kth-blue,
        line width = 1,
        rotate = 0,
        solid
      },forget plot] coordinates {
        (1.0, 0.9849624060150376)
        (3.0, 0.7357593984962406)
        (4.0, 0.5697443609022557)
        (6.0, 0.516812030075188)
        (7.0, 0.41022556390977444)
        (9.0, 0.39939849624060153)
        (10.0, 0.39554887218045115)
        (12.0, 0.31326315789473685)
        (14.0, 0.3067669172932331)
        (15.0, 0.35368421052631577)
        (17.0, 0.3498345864661654)
        (18.0, 0.35945864661654137)
        (20.0, 0.2699548872180451)
        (22.0, 0.2969022556390977)
        (23.0, 0.3041203007518797)
        (25.0, 0.29449624060150376)
        (26.0, 0.3387669172932331)
        (28.0, 0.26562406015037593)
        (29.0, 0.2709172932330827)
        (32.0, 0.13305263157894737)
      };
      \addplot+[draw=none, color = kth-red,
      line width = 0,
      solid,mark = triangle*,
      mark size = 3.0,
      mark options = {
        color = black,
        fill = kth-red,
        line width = 1,
        rotate = 0,
        solid
      }] coordinates {
        (1.0, 0.09859154929577464)
        (3.0, 0.20756115641215717)
        (4.0, 0.22090437361008153)
        (6.0, 0.2668643439584878)
        (7.0, 0.2542624166048925)
        (9.0, 0.30911786508524836)
        (10.0, 0.3061527057079318)
        (12.0, 0.32320237212750186)
        (14.0, 0.31653076352853965)
        (15.0, 0.3647146034099333)
        (17.0, 0.5404002965159377)
        (18.0, 0.5552260934025204)
        (20.0, 0.4173461823573017)
        (22.0, 0.4588584136397331)
        (23.0, 0.4699777613046701)
        (25.0, 0.4551519644180875)
        (26.0, 0.5233506300963677)
        (28.0, 0.4106745737583395)
        (29.0, 0.41882876204596)
        (32.0, 0.41141586360266863)
      };
      \addlegendentry{$N^{0.01}_I$}
      \addplot+ [color = kth-red,
      line width = 1,
      solid,mark = none,
      mark size = 2.0,
      mark options = {
        color = black,
        fill = kth-red,
        line width = 1,
        rotate = 0,
        solid
      },forget plot] coordinates {
        (1.0, 0.09859154929577464)
        (3.0, 0.20756115641215717)
        (4.0, 0.22090437361008153)
        (6.0, 0.2668643439584878)
        (7.0, 0.2542624166048925)
        (9.0, 0.30911786508524836)
        (10.0, 0.3061527057079318)
        (12.0, 0.32320237212750186)
        (14.0, 0.31653076352853965)
        (15.0, 0.3647146034099333)
        (17.0, 0.5404002965159377)
        (18.0, 0.5552260934025204)
        (20.0, 0.4173461823573017)
        (22.0, 0.4588584136397331)
        (23.0, 0.4699777613046701)
        (25.0, 0.4551519644180875)
        (26.0, 0.5233506300963677)
        (28.0, 0.4106745737583395)
        (29.0, 0.41882876204596)
        (32.0, 0.41141586360266863)
      };
      \addplot+[draw=none, color = kth-green,
      line width = 0,
      solid,mark = *,
      mark size = 2.0,
      mark options = {
        color = black,
        fill = kth-green,
        line width = 1,
        rotate = 0,
        solid
      }] coordinates {
        (1.0, 0.3949426867866125)
        (3.0, 0.42951542386510444)
        (4.0, 0.37740119498473707)
        (6.0, 0.4057778416436239)
        (7.0, 0.36894525714265897)
        (9.0, 0.42296698651258147)
        (10.0, 0.42629050634029536)
        (12.0, 0.39505131675011024)
        (14.0, 0.381354575493602)
        (15.0, 0.47199555177287966)
        (17.0, 0.6846567784430215)
        (18.0, 0.7308034304670608)
        (20.0, 0.4986681608828746)
        (22.0, 0.5714835338123356)
        (23.0, 0.5861152912876026)
        (25.0, 0.5522703887298046)
        (26.0, 0.6658001848789071)
        (28.0, 0.4688912312430288)
        (29.0, 0.48825334762623074)
        (32.0, 0.4352703023730862)
      };
      \addlegendentry{$N^{0.01}_T$}
      \addplot+ [color = kth-green,
      line width = 1,
      solid,mark = none,
      mark size = 2.0,
      mark options = {
        color = black,
        fill = kth-green,
        line width = 1,
        rotate = 0,
        solid
      },forget plot] coordinates {
        (1.0, 0.3949426867866125)
        (3.0, 0.42951542386510444)
        (4.0, 0.37740119498473707)
        (6.0, 0.4057778416436239)
        (7.0, 0.36894525714265897)
        (9.0, 0.42296698651258147)
        (10.0, 0.42629050634029536)
        (12.0, 0.39505131675011024)
        (14.0, 0.381354575493602)
        (15.0, 0.47199555177287966)
        (17.0, 0.6846567784430215)
        (18.0, 0.7308034304670608)
        (20.0, 0.4986681608828746)
        (22.0, 0.5714835338123356)
        (23.0, 0.5861152912876026)
        (25.0, 0.5522703887298046)
        (26.0, 0.6658001848789071)
        (28.0, 0.4688912312430288)
        (29.0, 0.48825334762623074)
        (32.0, 0.4352703023730862)
      };
    \end{axis}
  \end{tikzpicture}
}
  \caption{Relative empirical complexities when solving \texttt{20term} with $\num{1000}$ scenarios using dynamic aggregation under the \textbf{SelectUniform} rule, as a function of the parameter $T$.}
  \label{fig:20term_uniform}
\end{figure}

Next, we test the distance based aggregation schemes. We first consider dynamic aggregation under the \textbf{SelectClosest} rule, where the distance tolerance parameter $\tau$ is varied. From our experience, the angular distance measure introduced in~\ref{sec:distance-measures} is most reliably performant for all problems. Therefore, we only present the results from using this distance measure to keep the paper concise. The results are shown in Fig.~\ref{fig:20term_closest}. The same trade-off behaviour between cut- and iteration complexity can be observed as $\tau$ is increased. This is expected as the aggregation level should increase as the distance threshold for aggregation is increased. However, the time to converge is more parameter sensitive for this method. For come values, the \textbf{SelectClosest} scheme is even outperformed by full aggregation. Because the iteration complexity does not exhibit the same peaks, this increase in computation time is attributed to overhead in the procedure. The results from using cluster aggregation under the \textbf{Kmedoids} rule, shown in Fig.~\ref{fig:20term_kmedoids}, are more promising. The complexity trade-off is inversed for this scheme because the aggregation level increases as the number of possible clusters $k$ is increased. Any overhead from the k-medoids clustering calculations appears negligible. Even though some configurations of the partial cut approach converge faster the clustering approach is consistently efficient over a larger range of parameter values, indicating that the \textbf{Kmedoids} method could be easier to tune.

\begin{figure}
  \centering
  \resizebox{0.9\textwidth}{!}{
  \begin{tikzpicture}[]
    \begin{axis}[height = {53.269444444444446mm}, legend pos = {north west}, ylabel = {Relative complexity}, title = {\texttt{20term} - SelectClosest}, xmin = {0}, xmax = {1}, unbounded coords=jump,scaled x ticks = false, xlabel = {$\tau$}, xlabel style = {font = {\fontsize{14 pt}{14.3 pt}\selectfont}, rotate = 0.0},xmajorgrids = true, xtick align = inside,xticklabel style = {font = {\fontsize{12 pt}{10.4 pt}\selectfont}, rotate = 0.0},x grid style = {color = kth-lightgray,
        line width = 0.25,
        solid},axis x line* = left,x axis line style = {line width = 1,
        solid},scaled y ticks = false,ylabel style = {font = {\fontsize{14 pt}{14.3 pt}\selectfont}, rotate = 0.0},ymajorgrids = true,ytick align = inside,yticklabel style = {font = {\fontsize{12 pt}{10.4 pt}\selectfont}, rotate = 0.0},y grid style = {color = kth-lightgray,
        line width = 0.25,
        solid},axis y line* = left,y axis line style = {line width = 1,
        solid},    xshift = 0.0mm,
      yshift = 0mm,
      ,title style = {font = {\fontsize{14 pt}{18.2 pt}\selectfont}, rotate = 0.0},legend style = {line width = 1, xshift = 14mm,
        solid,font = {\fontsize{10 pt}{10.4 pt}\selectfont}},colorbar style={title=}, ymin = {0}, width = {152.4mm}]\addplot+[draw=none, color = kth-blue,
      line width = 0,
      solid,mark = square*,
      mark size = 2.0,
      mark options = {
        color = black,
        fill = kth-blue,
        line width = 1,
        rotate = 0,
        solid
      }] coordinates {
        (0.01, 0.7756315789473684)
        (0.06, 0.554796992481203)
        (0.11, 0.4658872180451128)
        (0.17, 0.3799924812030075)
        (0.22, 0.3045488721804511)
        (0.27, 0.24172932330827068)
        (0.32, 0.2583533834586466)
        (0.37, 0.22217293233082708)
        (0.43, 0.24366165413533836)
        (0.48, 0.23757142857142857)
        (0.53, 0.2214436090225564)
        (0.58, 0.2005187969924812)
        (0.64, 0.1519172932330827)
        (0.69, 0.22613533834586466)
        (0.74, 0.2154360902255639)
        (0.79, 0.16294736842105262)
        (0.84, 0.2036015037593985)
        (0.9, 0.1952481203007519)
        (0.95, 0.23075187969924812)
        (1.0, 0.221593984962406)
      };
      \addlegendentry{$N^{0.01}_C$}
      \addplot+ [color = kth-blue,
      line width = 1,
      solid,mark = none,
      mark size = 2.0,
      mark options = {
        color = black,
        fill = kth-blue,
        line width = 1,
        rotate = 0,
        solid
      },forget plot] coordinates {
        (0.01, 0.7756315789473684)
        (0.06, 0.554796992481203)
        (0.11, 0.4658872180451128)
        (0.17, 0.3799924812030075)
        (0.22, 0.3045488721804511)
        (0.27, 0.24172932330827068)
        (0.32, 0.2583533834586466)
        (0.37, 0.22217293233082708)
        (0.43, 0.24366165413533836)
        (0.48, 0.23757142857142857)
        (0.53, 0.2214436090225564)
        (0.58, 0.2005187969924812)
        (0.64, 0.1519172932330827)
        (0.69, 0.22613533834586466)
        (0.74, 0.2154360902255639)
        (0.79, 0.16294736842105262)
        (0.84, 0.2036015037593985)
        (0.9, 0.1952481203007519)
        (0.95, 0.23075187969924812)
        (1.0, 0.221593984962406)
      };
      \addplot+[draw=none, color = kth-red,
      line width = 0,
      solid,mark = triangle*,
      mark size = 3.0,
      mark options = {
        color = black,
        fill = kth-red,
        line width = 1,
        rotate = 0,
        solid
      }] coordinates {
        (0.01, 0.11045218680504076)
        (0.06, 0.15344699777613047)
        (0.11, 0.18680504077094143)
        (0.17, 0.2453669384729429)
        (0.22, 0.2824314306893996)
        (0.27, 0.312824314306894)
        (0.32, 0.3558191252779837)
        (0.37, 0.36249073387694586)
        (0.43, 0.4870274277242402)
        (0.48, 0.5107487027427724)
        (0.53, 0.5255744996293551)
        (0.58, 0.48851000741289846)
        (0.64, 0.346182357301705)
        (0.69, 0.6041512231282431)
        (0.74, 0.6204595997034841)
        (0.79, 0.45441067457375833)
        (0.84, 0.6212008895478132)
        (0.9, 0.6019273535952557)
        (0.95, 0.7086730911786508)
        (1.0, 0.6842105263157895)
      };
      \addlegendentry{$N^{0.01}_I$}
      \addplot+ [color = kth-red,
      line width = 1,
      solid,mark = none,
      mark size = 2.0,
      mark options = {
        color = black,
        fill = kth-red,
        line width = 1,
        rotate = 0,
        solid
      },forget plot] coordinates {
        (0.01, 0.11045218680504076)
        (0.06, 0.15344699777613047)
        (0.11, 0.18680504077094143)
        (0.17, 0.2453669384729429)
        (0.22, 0.2824314306893996)
        (0.27, 0.312824314306894)
        (0.32, 0.3558191252779837)
        (0.37, 0.36249073387694586)
        (0.43, 0.4870274277242402)
        (0.48, 0.5107487027427724)
        (0.53, 0.5255744996293551)
        (0.58, 0.48851000741289846)
        (0.64, 0.346182357301705)
        (0.69, 0.6041512231282431)
        (0.74, 0.6204595997034841)
        (0.79, 0.45441067457375833)
        (0.84, 0.6212008895478132)
        (0.9, 0.6019273535952557)
        (0.95, 0.7086730911786508)
        (1.0, 0.6842105263157895)
      };
      \addplot+[draw=none, color = kth-green,
      line width = 0,
      solid,mark = *,
      mark size = 2.0,
      mark options = {
        color = black,
        fill = kth-green,
        line width = 1,
        rotate = 0,
        solid
      }] coordinates {
        (0.01, 0.5263291315318477)
        (0.06, 0.3771438826521442)
        (0.11, 0.39217183940910616)
        (0.17, 0.4935569773027145)
        (0.22, 0.49077483719562187)
        (0.27, 0.551473472288401)
        (0.32, 0.6520651450843755)
        (0.37, 0.5941416478755174)
        (0.43, 0.9422909019066951)
        (0.48, 0.9624374003470579)
        (0.53, 0.8997680168809996)
        (0.58, 0.7597737531462853)
        (0.64, 0.49494520864017977)
        (0.69, 0.9975049333422207)
        (0.74, 0.8327587051002272)
        (0.79, 0.5637148444150496)
        (0.84, 0.7714465174601323)
        (0.9, 0.7575696699779957)
        (0.95, 0.9290954930753774)
        (1.0, 0.891362004797207)
      };
      \addlegendentry{$N^{0.01}_T$}
      \addplot+ [color = kth-green,
      line width = 1,
      solid,mark = none,
      mark size = 2.0,
      mark options = {
        color = black,
        fill = kth-green,
        line width = 1,
        rotate = 0,
        solid
      },forget plot] coordinates {
        (0.01, 0.5263291315318477)
        (0.06, 0.3771438826521442)
        (0.11, 0.39217183940910616)
        (0.17, 0.4935569773027145)
        (0.22, 0.49077483719562187)
        (0.27, 0.551473472288401)
        (0.32, 0.6520651450843755)
        (0.37, 0.5941416478755174)
        (0.43, 0.9422909019066951)
        (0.48, 0.9624374003470579)
        (0.53, 0.8997680168809996)
        (0.58, 0.7597737531462853)
        (0.64, 0.49494520864017977)
        (0.69, 0.9975049333422207)
        (0.74, 0.8327587051002272)
        (0.79, 0.5637148444150496)
        (0.84, 0.7714465174601323)
        (0.9, 0.7575696699779957)
        (0.95, 0.9290954930753774)
        (1.0, 0.891362004797207)
      };
    \end{axis}
  \end{tikzpicture}
}
  \caption{Relative empirical complexities when solving \texttt{20term} with $\num{1000}$ scenarios using dynamic aggregation under the \textbf{SelectClosest} rule, as a function of the parameter $\tau$. The angular distance measure is used throughout.}
  \label{fig:20term_closest}
\end{figure}
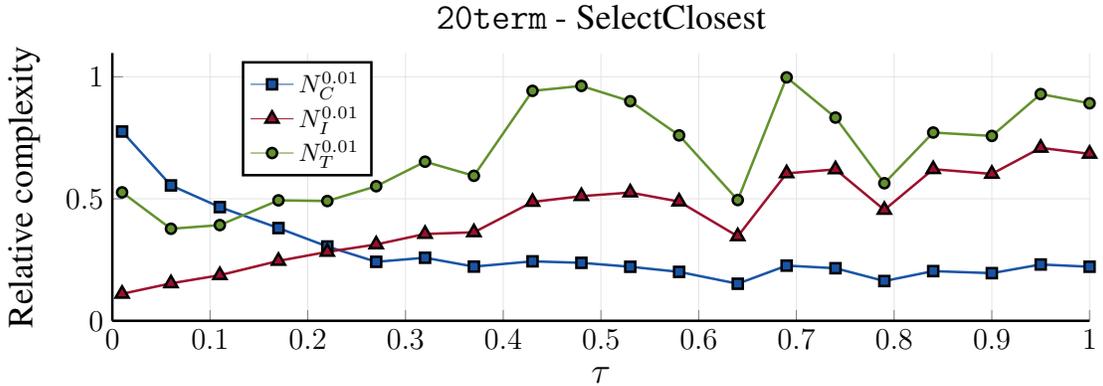

\begin{figure}
  \centering
  \resizebox{0.9\textwidth}{!}{
  \begin{tikzpicture}[]
    \begin{axis}[height = {53.269444444444446mm}, legend pos = {north west}, ylabel = {Relative complexity}, title = {\texttt{20term} - Kmedoids}, xmin = {0}, xmax = {32}, unbounded coords=jump,scaled x ticks = false, xlabel = {$k$}, xlabel style = {font = {\fontsize{14 pt}{14.3 pt}\selectfont}, rotate = 0.0},xmajorgrids = true, xtick align = inside,xticklabel style = {font = {\fontsize{12 pt}{10.4 pt}\selectfont}, rotate = 0.0},x grid style = {color = kth-lightgray,
        line width = 0.25,
        solid},axis x line* = left,x axis line style = {line width = 1,
        solid},scaled y ticks = false,ylabel style = {font = {\fontsize{14 pt}{14.3 pt}\selectfont}, rotate = 0.0},ymajorgrids = true,ytick align = inside,yticklabel style = {font = {\fontsize{12 pt}{10.4 pt}\selectfont}, rotate = 0.0},y grid style = {color = kth-lightgray,
        line width = 0.25,
        solid},axis y line* = left,y axis line style = {line width = 1,
        solid},    xshift = 0.0mm,
      yshift = 0mm,
      ymax = 1.4,
      ,title style = {font = {\fontsize{14 pt}{18.2 pt}\selectfont}, rotate = 0.0},legend style = {line width = 1,
        solid,font = {\fontsize{10 pt}{10.4 pt}\selectfont}},colorbar style={title=}, ymin = {0}, width = {152.4mm}]\addplot+[draw=none, color = kth-blue,
      line width = 0,
      solid,mark = square*,
      mark size = 2.0,
      mark options = {
        color = black,
        fill = kth-blue,
        line width = 1,
        rotate = 0,
        solid
      }] coordinates {
        (1.0, 0.18213533834586465)
        (3.0, 0.21109022556390977)
        (4.0, 0.27269172932330826)
        (6.0, 0.27473684210526317)
        (7.0, 0.34598496240601506)
        (9.0, 0.39246616541353385)
        (10.0, 0.43362406015037597)
        (12.0, 0.4706766917293233)
        (14.0, 0.5006992481203008)
        (15.0, 0.47741353383458646)
        (17.0, 0.5858345864661654)
        (18.0, 0.6364285714285715)
        (20.0, 0.6704962406015038)
        (22.0, 0.7616992481203008)
        (23.0, 0.7129398496240601)
        (25.0, 0.7538496240601503)
        (26.0, 0.816766917293233)
        (28.0, 0.8822932330827068)
        (29.0, 0.9595789473684211)
        (31.0, 0.8684436090225564)
      };
      \addlegendentry{$N^{0.01}_C$}
      \addplot+ [color = kth-blue,
      line width = 1,
      solid,mark = none,
      mark size = 2.0,
      mark options = {
        color = black,
        fill = kth-blue,
        line width = 1,
        rotate = 0,
        solid
      },forget plot] coordinates {
        (1.0, 0.18213533834586465)
        (3.0, 0.21109022556390977)
        (4.0, 0.27269172932330826)
        (6.0, 0.27473684210526317)
        (7.0, 0.34598496240601506)
        (9.0, 0.39246616541353385)
        (10.0, 0.43362406015037597)
        (12.0, 0.4706766917293233)
        (14.0, 0.5006992481203008)
        (15.0, 0.47741353383458646)
        (17.0, 0.5858345864661654)
        (18.0, 0.6364285714285715)
        (20.0, 0.6704962406015038)
        (22.0, 0.7616992481203008)
        (23.0, 0.7129398496240601)
        (25.0, 0.7538496240601503)
        (26.0, 0.816766917293233)
        (28.0, 0.8822932330827068)
        (29.0, 0.9595789473684211)
        (31.0, 0.8684436090225564)
      };
      \addplot+[draw=none, color = kth-red,
      line width = 0,
      solid,mark = triangle*,
      mark size = 3.0,
      mark options = {
        color = black,
        fill = kth-red,
        line width = 1,
        rotate = 0,
        solid
      }] coordinates {
        (1.0, 0.5626389918458117)
        (3.0, 0.2246108228317272)
        (4.0, 0.2149740548554485)
        (6.0, 0.14381022979985175)
        (7.0, 0.15492957746478872)
        (9.0, 0.13713862120088954)
        (10.0, 0.1363973313565604)
        (12.0, 0.12305411415863603)
        (14.0, 0.11267605633802817)
        (15.0, 0.10081541882876205)
        (17.0, 0.10822831727205337)
        (18.0, 0.1111934766493699)
        (20.0, 0.10526315789473684)
        (22.0, 0.11045218680504076)
        (23.0, 0.09933283914010378)
        (25.0, 0.10081541882876205)
        (26.0, 0.10378057820607858)
        (28.0, 0.10600444773906598)
        (29.0, 0.1156412157153447)
        (31.0, 0.08969607116382505)
      };
      \addlegendentry{$N^{0.01}_I$}
      \addplot+ [color = kth-red,
      line width = 1,
      solid,mark = none,
      mark size = 2.0,
      mark options = {
        color = black,
        fill = kth-red,
        line width = 1,
        rotate = 0,
        solid
      },forget plot] coordinates {
        (1.0, 0.5626389918458117)
        (3.0, 0.2246108228317272)
        (4.0, 0.2149740548554485)
        (6.0, 0.14381022979985175)
        (7.0, 0.15492957746478872)
        (9.0, 0.13713862120088954)
        (10.0, 0.1363973313565604)
        (12.0, 0.12305411415863603)
        (14.0, 0.11267605633802817)
        (15.0, 0.10081541882876205)
        (17.0, 0.10822831727205337)
        (18.0, 0.1111934766493699)
        (20.0, 0.10526315789473684)
        (22.0, 0.11045218680504076)
        (23.0, 0.09933283914010378)
        (25.0, 0.10081541882876205)
        (26.0, 0.10378057820607858)
        (28.0, 0.10600444773906598)
        (29.0, 0.1156412157153447)
        (31.0, 0.08969607116382505)
      };
      \addplot+[draw=none, color = kth-green,
      line width = 0,
      solid,mark = *,
      mark size = 2.0,
      mark options = {
        color = black,
        fill = kth-green,
        line width = 1,
        rotate = 0,
        solid
      }] coordinates {
        (1.0, 0.6131633399593703)
        (3.0, 0.5575894599461901)
        (4.0, 0.7029033523176512)
        (6.0, 0.5786625826856414)
        (7.0, 0.6783799252171973)
        (9.0, 0.6413728101412107)
        (10.0, 0.6348273621627823)
        (12.0, 0.5986707059956675)
        (14.0, 0.46232827511493507)
        (15.0, 0.3786385383446904)
        (17.0, 0.3644303173065756)
        (18.0, 0.40074484204275246)
        (20.0, 0.3437874124131079)
        (22.0, 0.6221616510744243)
        (23.0, 0.34986556293203785)
        (25.0, 0.3587549524057021)
        (26.0, 0.37217664982105536)
        (28.0, 0.4410468194142772)
        (29.0, 0.5008163661568644)
        (31.0, 0.4310929382184468)
      };
      \addlegendentry{$N^{0.01}_T$}
      \addplot+ [color = kth-green,
      line width = 1,
      solid,mark = none,
      mark size = 2.0,
      mark options = {
        color = black,
        fill = kth-green,
        line width = 1,
        rotate = 0,
        solid
      },forget plot] coordinates {
        (1.0, 0.6131633399593703)
        (3.0, 0.5575894599461901)
        (4.0, 0.7029033523176512)
        (6.0, 0.5786625826856414)
        (7.0, 0.6783799252171973)
        (9.0, 0.6413728101412107)
        (10.0, 0.6348273621627823)
        (12.0, 0.5986707059956675)
        (14.0, 0.46232827511493507)
        (15.0, 0.3786385383446904)
        (17.0, 0.3644303173065756)
        (18.0, 0.40074484204275246)
        (20.0, 0.3437874124131079)
        (22.0, 0.6221616510744243)
        (23.0, 0.34986556293203785)
        (25.0, 0.3587549524057021)
        (26.0, 0.37217664982105536)
        (28.0, 0.4410468194142772)
        (29.0, 0.5008163661568644)
        (31.0, 0.4310929382184468)
      };
    \end{axis}
  \end{tikzpicture}
}
  \caption{Relative empirical complexities when solving \texttt{20term} with $\num{1000}$ scenarios using cluster aggregation under the \textbf{Kmedoids} rule, as a function of the parameter $k$. The angular distance measure is used throughout.}
  \label{fig:20term_kmedoids}
\end{figure}
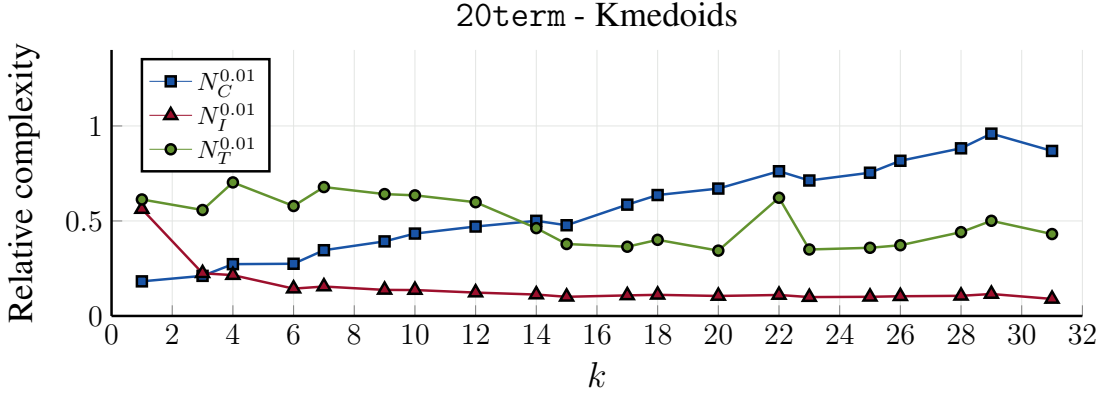

Finally, we test if the performance of the novel aggregation schemes can be improved using the suggested granulated approach. We use an initial static partitioning of size $\num{3}$. Each worker will then run dynamic aggregation schemes on about $\num{10}$ pre-granulated aggregates each. The result for the \textbf{SelectClosest} rule is shown in Fig.~\ref{fig:20term_granulated_closest}. The overhead apparent in the nominal implementation is successfully removed and the time complexity appears iteration bound instead. Overall performance is also improved compared to the non-granulated version, but it is still not competitive with partial cut aggregation. The results of running granulated \textbf{Kmedoids} cluster aggregation is shown in Fig.~\ref{fig:20term_granulated_kmedoids}. These results are more promising. Performance is again improved overall compared to not using pre-granulation. Moreover, this method outperforms all other aggregation schemes in this small-scale setting. Compared to the other aggregation schemes, both of the granulated methods show more consistent performance as the relevant aggregation parameter is varied. The number of possible aggregation combinations is naturally reduced by the pre-granulation, which makes the performance less sensitive to the parameter value. In addition, the \textbf{Kmedoids} strategy is also able to further improve performance from only using static aggregation. This is an indication that the granulated strategy is successful in combining the strengths of static and dynamic aggregation.

\begin{figure}
  \centering
  \resizebox{0.9\textwidth}{!}{
  \begin{tikzpicture}[]
    \begin{axis}[height = {53.269444444444446mm}, legend pos = {north west}, ylabel = {Complexity fraction}, title = {\texttt{20term} - GranulatedSelectClosest}, xmin = {0}, xmax = {1}, unbounded coords=jump,scaled x ticks = false, xlabel = {$\tau
        $}, xlabel style = {font = {\fontsize{14 pt}{14.3 pt}\selectfont}, rotate = 0.0},xmajorgrids = true, xtick align = inside,xticklabel style = {font = {\fontsize{12 pt}{10.4 pt}\selectfont}, rotate = 0.0},x grid style = {color = kth-lightgray,
        line width = 0.25,
        solid},axis x line* = left,x axis line style = {line width = 1,
        solid},scaled y ticks = false,ylabel style = {font = {\fontsize{14 pt}{14.3 pt}\selectfont}, rotate = 0.0},ymajorgrids = true,ytick align = inside,yticklabel style = {font = {\fontsize{12 pt}{10.4 pt}\selectfont}, rotate = 0.0},y grid style = {color = kth-lightgray,
        line width = 0.25,
        solid},axis y line* = left,y axis line style = {line width = 1,
        solid},    xshift = 0.0mm,
      yshift = 0mm,
      ymax = 1,
      ,title style = {font = {\fontsize{14 pt}{18.2 pt}\selectfont}, rotate = 0.0},legend style = {line width = 1, xshift = 4mm,
        solid,font = {\fontsize{10 pt}{10.4 pt}\selectfont}},colorbar style={title=}, ymin = {0}, width = {152.4mm}]\addplot+[draw=none, color = kth-blue,
      line width = 0,
      solid,mark = square*,
      mark size = 2.0,
      mark options = {
        color = black,
        fill = kth-blue,
        line width = 1,
        rotate = 0,
        solid
      }] coordinates {
        (0.01, 0.5996315789473684)
        (0.06, 0.45150375939849624)
        (0.11, 0.31301503759398497)
        (0.17, 0.27866165413533833)
        (0.22, 0.2561578947368421)
        (0.27, 0.15798496240601503)
        (0.32, 0.15382706766917292)
        (0.37, 0.23257142857142857)
        (0.43, 0.19690225563909775)
        (0.48, 0.17906766917293232)
        (0.53, 0.2321203007518797)
        (0.58, 0.23548872180451128)
        (0.64, 0.2051578947368421)
        (0.69, 0.21805263157894736)
        (0.74, 0.23242105263157894)
        (0.79, 0.20258646616541354)
        (0.84, 0.14027067669172932)
        (0.9, 0.15879699248120302)
        (0.95, 0.16769924812030076)
        (1.0, 0.17996992481203009)
      };
      \addlegendentry{$N^{0.01}_C$}
      \addplot+ [color = kth-blue,
      line width = 1,
      solid,mark = none,
      mark size = 2.0,
      mark options = {
        color = black,
        fill = kth-blue,
        line width = 1,
        rotate = 0,
        solid
      },forget plot] coordinates {
        (0.01, 0.5996315789473684)
        (0.06, 0.45150375939849624)
        (0.11, 0.31301503759398497)
        (0.17, 0.27866165413533833)
        (0.22, 0.2561578947368421)
        (0.27, 0.15798496240601503)
        (0.32, 0.15382706766917292)
        (0.37, 0.23257142857142857)
        (0.43, 0.19690225563909775)
        (0.48, 0.17906766917293232)
        (0.53, 0.2321203007518797)
        (0.58, 0.23548872180451128)
        (0.64, 0.2051578947368421)
        (0.69, 0.21805263157894736)
        (0.74, 0.23242105263157894)
        (0.79, 0.20258646616541354)
        (0.84, 0.14027067669172932)
        (0.9, 0.15879699248120302)
        (0.95, 0.16769924812030076)
        (1.0, 0.17996992481203009)
      };
      \addplot+[draw=none, color = kth-red,
      line width = 0,
      solid,mark = triangle*,
      mark size = 3.0,
      mark options = {
        color = black,
        fill = kth-red,
        line width = 1,
        rotate = 0,
        solid
      }] coordinates {
        (0.01, 0.2223869532987398)
        (0.06, 0.3061527057079318)
        (0.11, 0.37064492216456635)
        (0.17, 0.4447739065974796)
        (0.22, 0.5114899925871016)
        (0.27, 0.37064492216456635)
        (0.32, 0.38028169014084506)
        (0.37, 0.6597479614529281)
        (0.43, 0.5700518902891031)
        (0.48, 0.5396590066716086)
        (0.53, 0.7094143810229799)
        (0.58, 0.7242401779095626)
        (0.64, 0.631578947368421)
        (0.69, 0.6723498888065234)
        (0.74, 0.7175685693106004)
        (0.79, 0.625648628613788)
        (0.84, 0.4336545589325426)
        (0.9, 0.4907338769458858)
        (0.95, 0.5181616011860638)
        (1.0, 0.5559673832468495)
      };
      \addlegendentry{$N^{0.01}_I$}
      \addplot+ [color = kth-red,
      line width = 1,
      solid,mark = none,
      mark size = 2.0,
      mark options = {
        color = black,
        fill = kth-red,
        line width = 1,
        rotate = 0,
        solid
      },forget plot] coordinates {
        (0.01, 0.2223869532987398)
        (0.06, 0.3061527057079318)
        (0.11, 0.37064492216456635)
        (0.17, 0.4447739065974796)
        (0.22, 0.5114899925871016)
        (0.27, 0.37064492216456635)
        (0.32, 0.38028169014084506)
        (0.37, 0.6597479614529281)
        (0.43, 0.5700518902891031)
        (0.48, 0.5396590066716086)
        (0.53, 0.7094143810229799)
        (0.58, 0.7242401779095626)
        (0.64, 0.631578947368421)
        (0.69, 0.6723498888065234)
        (0.74, 0.7175685693106004)
        (0.79, 0.625648628613788)
        (0.84, 0.4336545589325426)
        (0.9, 0.4907338769458858)
        (0.95, 0.5181616011860638)
        (1.0, 0.5559673832468495)
      };
      \addplot+[draw=none, color = kth-green,
      line width = 0,
      solid,mark = *,
      mark size = 2.0,
      mark options = {
        color = black,
        fill = kth-green,
        line width = 1,
        rotate = 0,
        solid
      }] coordinates {
        (0.01, 0.4186948576431456)
        (0.06, 0.4158718111547028)
        (0.11, 0.420719791841106)
        (0.17, 0.48669908567678255)
        (0.22, 0.5672339384951046)
        (0.27, 0.3680280610455118)
        (0.32, 0.36832405561313863)
        (0.37, 0.701022148010766)
        (0.43, 0.5640509055352546)
        (0.48, 0.5167738698254994)
        (0.53, 0.7566649429686375)
        (0.58, 0.7750890291536937)
        (0.64, 0.6462556090384054)
        (0.69, 0.7104164761413686)
        (0.74, 0.7614452206792944)
        (0.79, 0.6521119789582013)
        (0.84, 0.3967625651232074)
        (0.9, 0.4563492148571285)
        (0.95, 0.4960006371675924)
        (1.0, 0.5389514773414767)
      };
      \addlegendentry{$N^{0.01}_T$}
      \addplot+ [color = kth-green,
      line width = 1,
      solid,mark = none,
      mark size = 2.0,
      mark options = {
        color = black,
        fill = kth-green,
        line width = 1,
        rotate = 0,
        solid
      },forget plot] coordinates {
        (0.01, 0.4186948576431456)
        (0.06, 0.4158718111547028)
        (0.11, 0.420719791841106)
        (0.17, 0.48669908567678255)
        (0.22, 0.5672339384951046)
        (0.27, 0.3680280610455118)
        (0.32, 0.36832405561313863)
        (0.37, 0.701022148010766)
        (0.43, 0.5640509055352546)
        (0.48, 0.5167738698254994)
        (0.53, 0.7566649429686375)
        (0.58, 0.7750890291536937)
        (0.64, 0.6462556090384054)
        (0.69, 0.7104164761413686)
        (0.74, 0.7614452206792944)
        (0.79, 0.6521119789582013)
        (0.84, 0.3967625651232074)
        (0.9, 0.4563492148571285)
        (0.95, 0.4960006371675924)
        (1.0, 0.5389514773414767)
      };
    \end{axis}
  \end{tikzpicture}
}
  \caption{Relative empirical complexities when solving \texttt{20term} with $\num{1000}$ scenarios using granulated aggregation of size $3$ followed by dynamic aggregation under the \textbf{SelectClosest} rule, as a function of the parameter $\tau$. The angular distance measure is used throughout.}
  \label{fig:20term_granulated_closest}
\end{figure}
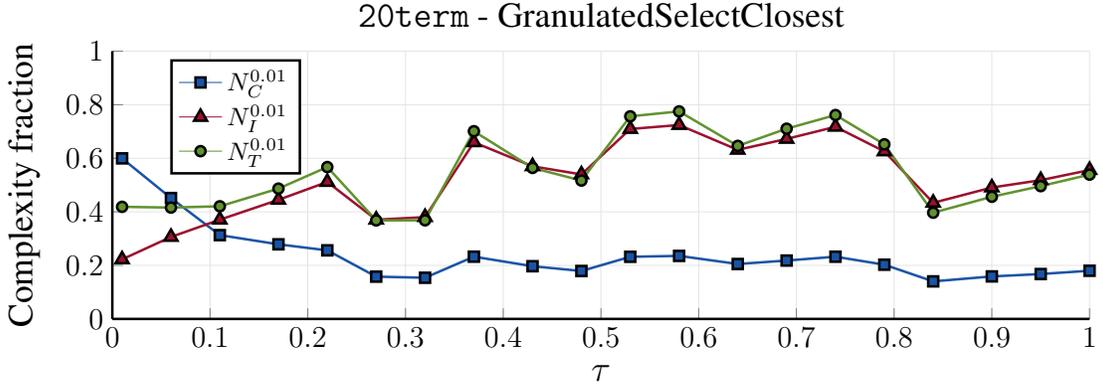

\begin{figure}
  \centering
  \resizebox{0.9\textwidth}{!}{
  \begin{tikzpicture}[]
    \begin{axis}[height = {53.269444444444446mm}, legend pos = {north west}, ylabel = {Complexity fraction}, title = {\texttt{20term} - GranulatedKmedoids}, xmin = {0}, xmax = {32}, unbounded coords=jump,scaled x ticks = false, xlabel = {$k$}, xlabel style = {font = {\fontsize{14 pt}{14.3 pt}\selectfont}, rotate = 0.0},xmajorgrids = true, xtick align = inside,xticklabel style = {font = {\fontsize{12 pt}{10.4 pt}\selectfont}, rotate = 0.0},x grid style = {color = kth-lightgray,
        line width = 0.25,
        solid},axis x line* = left,x axis line style = {line width = 1,
        solid},scaled y ticks = false,ylabel style = {font = {\fontsize{14 pt}{14.3 pt}\selectfont}, rotate = 0.0},ymajorgrids = true,ytick align = inside,yticklabel style = {font = {\fontsize{12 pt}{10.4 pt}\selectfont}, rotate = 0.0},y grid style = {color = kth-lightgray,
        line width = 0.25,
        solid},axis y line* = left,y axis line style = {line width = 1,
        solid},    xshift = 0.0mm,
      yshift = 0mm,
      ymax = 1.0,
      ,title style = {font = {\fontsize{14 pt}{18.2 pt}\selectfont}, rotate = 0.0},legend style = {line width = 1, xshift = 8mm,
        solid,font = {\fontsize{10 pt}{10.4 pt}\selectfont}},colorbar style={title=}, ymin = {0}, width = {152.4mm}]\addplot+[draw=none, color = kth-blue,
      line width = 0,
      solid,mark = square*,
      mark size = 2.0,
      mark options = {
        color = black,
        fill = kth-blue,
        line width = 1,
        rotate = 0,
        solid
      }] coordinates {
        (1.0, 0.2081203007518797)
        (3.0, 0.27446616541353386)
        (4.0, 0.24252631578947367)
        (6.0, 0.4056541353383459)
        (7.0, 0.4301954887218045)
        (9.0, 0.5695037593984963)
        (10.0, 0.5889924812030075)
        (12.0, 0.6219548872180451)
        (14.0, 0.772812030075188)
        (15.0, 0.614015037593985)
        (17.0, 0.6748872180451128)
        (18.0, 0.6643007518796993)
        (20.0, 0.6881203007518797)
        (22.0, 0.690766917293233)
        (23.0, 0.7172330827067669)
        (25.0, 0.6643007518796993)
        (26.0, 0.5002105263157894)
        (28.0, 0.7278195488721805)
        (29.0, 0.7463458646616541)
        (31.0, 0.5849022556390977)
      };
      \addlegendentry{$N^{0.01}_C$}
      \addplot+ [color = kth-blue,
      line width = 1,
      solid,mark = none,
      mark size = 2.0,
      mark options = {
        color = black,
        fill = kth-blue,
        line width = 1,
        rotate = 0,
        solid
      },forget plot] coordinates {
        (1.0, 0.2081203007518797)
        (3.0, 0.27446616541353386)
        (4.0, 0.24252631578947367)
        (6.0, 0.4056541353383459)
        (7.0, 0.4301954887218045)
        (9.0, 0.5695037593984963)
        (10.0, 0.5889924812030075)
        (12.0, 0.6219548872180451)
        (14.0, 0.772812030075188)
        (15.0, 0.614015037593985)
        (17.0, 0.6748872180451128)
        (18.0, 0.6643007518796993)
        (20.0, 0.6881203007518797)
        (22.0, 0.690766917293233)
        (23.0, 0.7172330827067669)
        (25.0, 0.6643007518796993)
        (26.0, 0.5002105263157894)
        (28.0, 0.7278195488721805)
        (29.0, 0.7463458646616541)
        (31.0, 0.5849022556390977)
      };
      \addplot+[draw=none, color = kth-red,
      line width = 0,
      solid,mark = triangle*,
      mark size = 3.0,
      mark options = {
        color = black,
        fill = kth-red,
        line width = 1,
        rotate = 0,
        solid
      }] coordinates {
        (1.0, 0.6426982950333581)
        (3.0, 0.28465530022238694)
        (4.0, 0.1882876204595997)
        (6.0, 0.20978502594514456)
        (7.0, 0.19125277983691624)
        (9.0, 0.19718309859154928)
        (10.0, 0.18383988139362492)
        (12.0, 0.17568569310600446)
        (14.0, 0.21793921423276502)
        (15.0, 0.17346182357301704)
        (17.0, 0.1905114899925871)
        (18.0, 0.18754633061527057)
        (20.0, 0.19421793921423278)
        (22.0, 0.1949592290585619)
        (23.0, 0.20237212750185324)
        (25.0, 0.18754633061527057)
        (26.0, 0.14158636026686433)
        (28.0, 0.20533728687916974)
        (29.0, 0.21052631578947367)
        (31.0, 0.1653076352853966)
      };
      \addlegendentry{$N^{0.01}_I$}
      \addplot+ [color = kth-red,
      line width = 1,
      solid,mark = none,
      mark size = 2.0,
      mark options = {
        color = black,
        fill = kth-red,
        line width = 1,
        rotate = 0,
        solid
      },forget plot] coordinates {
        (1.0, 0.6426982950333581)
        (3.0, 0.28465530022238694)
        (4.0, 0.1882876204595997)
        (6.0, 0.20978502594514456)
        (7.0, 0.19125277983691624)
        (9.0, 0.19718309859154928)
        (10.0, 0.18383988139362492)
        (12.0, 0.17568569310600446)
        (14.0, 0.21793921423276502)
        (15.0, 0.17346182357301704)
        (17.0, 0.1905114899925871)
        (18.0, 0.18754633061527057)
        (20.0, 0.19421793921423278)
        (22.0, 0.1949592290585619)
        (23.0, 0.20237212750185324)
        (25.0, 0.18754633061527057)
        (26.0, 0.14158636026686433)
        (28.0, 0.20533728687916974)
        (29.0, 0.21052631578947367)
        (31.0, 0.1653076352853966)
      };
      \addplot+[draw=none, color = kth-green,
      line width = 0,
      solid,mark = *,
      mark size = 2.0,
      mark options = {
        color = black,
        fill = kth-green,
        line width = 1,
        rotate = 0,
        solid
      }] coordinates {
        (1.0, 0.6551516235677696)
        (3.0, 0.37614197439325386)
        (4.0, 0.23788146698201357)
        (6.0, 0.3257821950061913)
        (7.0, 0.3071995506122625)
        (9.0, 0.3519028535848305)
        (10.0, 0.3124699499331115)
        (12.0, 0.32423223387060796)
        (14.0, 0.4496253565271617)
        (15.0, 0.35659636266675815)
        (17.0, 0.4074201345806001)
        (18.0, 0.37072615682694043)
        (20.0, 0.3749298275028284)
        (22.0, 0.3664823780627014)
        (23.0, 0.37091286641396515)
        (25.0, 0.3423501322070544)
        (26.0, 0.24704050465263794)
        (28.0, 0.39354330305471946)
        (29.0, 0.4055756328748692)
        (31.0, 0.29858229608493786)
      };
      \addlegendentry{$N^{0.01}_T$}
      \addplot+ [color = kth-green,
      line width = 1,
      solid,mark = none,
      mark size = 2.0,
      mark options = {
        color = black,
        fill = kth-green,
        line width = 1,
        rotate = 0,
        solid
      },forget plot] coordinates {
        (1.0, 0.6551516235677696)
        (3.0, 0.37614197439325386)
        (4.0, 0.23788146698201357)
        (6.0, 0.3257821950061913)
        (7.0, 0.3071995506122625)
        (9.0, 0.3519028535848305)
        (10.0, 0.3124699499331115)
        (12.0, 0.32423223387060796)
        (14.0, 0.4496253565271617)
        (15.0, 0.35659636266675815)
        (17.0, 0.4074201345806001)
        (18.0, 0.37072615682694043)
        (20.0, 0.3749298275028284)
        (22.0, 0.3664823780627014)
        (23.0, 0.37091286641396515)
        (25.0, 0.3423501322070544)
        (26.0, 0.24704050465263794)
        (28.0, 0.39354330305471946)
        (29.0, 0.4055756328748692)
        (31.0, 0.29858229608493786)
      };
    \end{axis}
  \end{tikzpicture}
}
  \caption{Relative empirical complexities when solving \texttt{20term} with $\num{1000}$ scenarios using granulated aggregation of size $3$ followed by cluster aggregation under the \textbf{Kmedoids} rule, as a function of the parameter $k$. The angular distance measure is used throughout.}
  \label{fig:20term_granulated_kmedoids}
\end{figure}
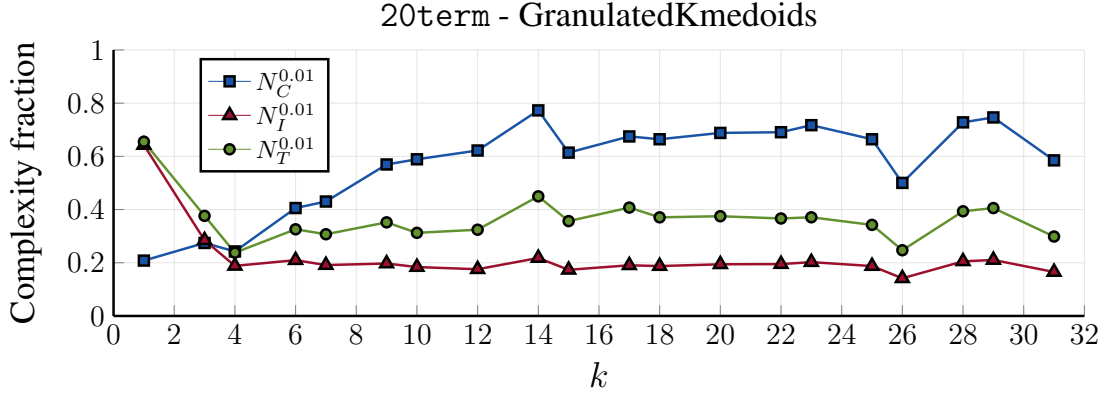

To facilitate a clear comparison between the aggregation scheme performances, the wall-clock time required to converge is shown for all methods in Fig.~\ref{fig:small_scale}. It is apparent that \textbf{SelectClosest} is not competitive with the other methods for this problem. \textbf{Kmedoids} starts being competitive with uniform aggregation schemes for larger values of $k$. \textbf{GranulatedKmedoids} outperforms the other methods fairly consistently and also achieves the shortest time to solution.

\begin{figure}
  \centering
  \input{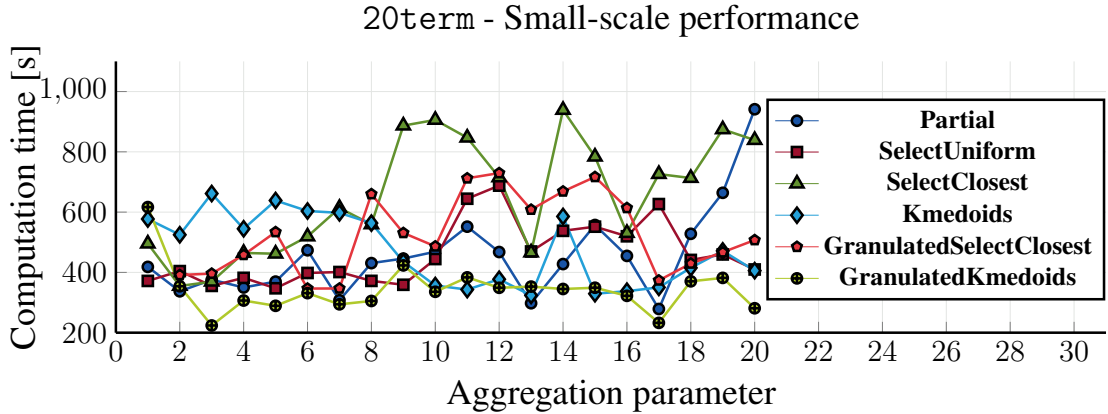}
  \caption{Wall-clock time required to converge within a relative tolerance of $\num{1e-2}$ when solving \texttt{20term} with $\num{1000}$ scenarios, using L-shaped with different aggregation schemes. Because the relevant aggregation parameters differ between the methods the dependent axis is an enumeration of the parameters.}
  \label{fig:small_scale}
\end{figure}

\subsection{Large-scale experiments}
\label{sec:large-scale-exper}

Now, we benchmark all problems in Table~\ref{tab:test_problems} at the intended sample sizes. To reduce random measurement noise we solve the same sampled problem five times and report median computation times. We prefer medians over averages because they are not as skewed by outliers. The variations in the measurements are relatively small compared to the computation times, so five repeats are deemed enough to reduce random errors. For comparison, we also measure the time required to solve the deterministic equivalents directly with Gurobi~\cite{gurobi}. The small-scale experiments indicate that most aggregation schemes are sensitive to the parameter configurations. It is not feasible to tune parameters when solving large-scale problems. Instead, our aim is to strive for perfomance improvements with generally applicable parameter settings. That is, we settle for a set of parameter configurations and apply them on all test problems.

The small-scale experiments show that the optimal aggregation level $T$ for uniform aggregation schemes is hard to guess, but that smaller values appear more performant than larger values. Therefore, we fix the aggregation level to $10\%$ of the number of scenarios on each worker for both \textbf{Partial} and \textbf{SelectUniform}. For example, $T = 16$ for the problems with sample size $\num{5000}$. For \textbf{SelectClosest}, smaller values of $\tau$ appear more performant. We try the value $\tau = 0.3$ on all problems. \textbf{Kmedoids} appeared less sensitive to the $k$ parameter, but in general a larger number of clusters gave better performance. We fix the value $k = 20$ in all experiments. For the granulated aggregation schemes we include two variations. First, we devise an aggressive scheme with a coarse initial partitioning that matches the size of the uniform schemes. The scheme \textbf{CoarseGranulatedSelectClosest} is thus a pre-granulation of size $\frac{0.1N}{32}$ followed by \textbf{SelectClosest} with $\tau = 0.3$. Likewise, the scheme \textbf{CoarseGranulatedKmedoids} is a pre-granulation of size $\frac{0.1N}{32}$ followed by \textbf{Kmedoids} with $k = 5$. We use a small value of $k$ because this gave the best performance in the small-scale test of \textbf{GranulatedKmedoids}. In addition, we devise a less aggressive scheme with a more fine initial partitioning of less than $1\%$ of the sample size. The aim is to reduce any scalability issues of the dynamic methods and otherwise use the same configurations. The parameter configurations used in the experiments are summarized in Table~\ref{tab:parameters}.

\begin{table}[htbp]
  \centering
  \begin{tabular}{|c|ccc|}
    \hline
    \backslashbox{Algorithm}{Problem} & \texttt{LandS}/\texttt{gbd} & \texttt{20term}/\texttt{ssn}/\texttt{storm} & \texttt{dayahead} \\
    \hline
    \textbf{Partial} & $T = 312$ & $T = 16$ & $T = 3$ \\
    \textbf{SelectUniform} & $T = 312$ & $T = 16$ & $T = 3$ \\
    \textbf{SelectClosest} & $\tau = 0.3$ & $\tau = 0.3$ & $\tau = 0.3$ \\
    \textbf{Kmedoids} & $k = 20$ & $k = 20$ & $k = 20$ \\
    \textbf{CoarseGranulatedSelectClosest} & $T = 312,\; \tau = 0.3$ & $T = 16,\; \tau = 0.3$ & $T = 5,\; \tau = 0.3$ \\
    \textbf{CoarseGranulatedKmedoids} & $T = 312,\; k = 5$ & $T = 16,\; k = 5$ & $T = 5,\; k = 5$ \\
    \textbf{FineGranulatedSelectClosest} & $T = 100,\; \tau = 0.3$ & $T = 5,\; \tau = 0.3$ & $T = 3,\; \tau = 0.3$ \\
    \textbf{FineGranulatedKmedoids} & $T = 100,\; k = 20$ & $T = 5,\; k = 20$ & $T = 3,\; k = 20$ \\
    \hline
  \end{tabular}
  \caption{Parameter configurations for each aggregation scheme used in the large-scale experiments. Problems that use the same configurations are grouped together.}
  \label{tab:parameters}
\end{table}

The results of the large-scale experiment are presented in Table~\ref{tab:largescale}. For easier comparison we also illustrate the results in Fig.~\ref{fig:large_scale}. As in initial observation, granulated aggregation schemes yields the best performance in five out of six problems. The \texttt{storm} problem is notorious for its flat objective that make L-shaped algorithms hard to tune~\cite{linderoth_decomposition_2003, Trukhanov2010}. Independent of the aggregation scheme, an optimal solution is found after $10-11$ iterations. Hence, the overhead from more intricate aggregation strategies yield worse performance than partial cut aggregation due to overhead. Any type of aggregation yields better performance than multi-cut on these large-scale problems. The time taken to solve the deterministic equivalents clarify the need for distributed approaches when solving large-scale stochastic programs. Most of the solution times are spent constructing the deterministic equivalent in memory. The \texttt{LandS} and \texttt{gbd} problems that have the largest number of scenarios do not even finish building the deterministic equivalent in reasonable time. The scalability issues of the dynamic schemes are more prominent than in the small-scale test. It could also hold that the parameter choices for \textbf{SelectClosest} and \textbf{Kmedoids} are suboptimal. However, these issues are significantly reduced by using granulated strategies.

\begin{table}[htbp]
  \centering
  \begin{tabular}{|c|cccccc|}
    \hline
    \backslashbox{Algorithm}{Wall-clock time [s]}{Problem} & \texttt{LandS} & \texttt{gbd} & \texttt{20term} & \texttt{ssn} & \texttt{storm} & \texttt{dayahead} \\
    \hline
    \textbf{Deterministic} & - & - & 5939.9 & 2944.3 & 5394.1 & 1053.2 \\
    \textbf{Multi-cut} & 2867.9 & 2910.1 & 4541.2 & 593.4 & 654.7 & 119.3 \\
    \textbf{Partial} & 107.5 & 51.4 & 1076.9 & 512.2 & \textbf{552.0} & 97.2 \\
    \textbf{SelectUniform} & 145.6 & 77.1 & 1619.2 & 601.7 & 559.9 & 95.7 \\
    \textbf{SelectClosest} & 171.4 & 217.1 & 6284.7 & 3028.8 & 775.8 & 114.5 \\
    \textbf{Kmedoids} & 248.7 & 191.5 & 11785.1 & 4273.5 & 2082.3 & 114.6 \\
    \textbf{CoarseGranulatedSelectClosest} & \textbf{103.3} & \textbf{42.6} & 1615.7 & 844.7 & 774.5 & 109.7 \\
    \textbf{CoarseGranulatedKmedoids} & 126.3 & 44.8 & \textbf{779.2} & 627.6 & 780.2 & 119.5 \\
    \textbf{FineGranulatedSelectClosest} & 109.3 & 42.7 & 1558.7 & 443.4 & 771.0 & 107.3 \\
    \textbf{FineGranulatedKmedoids} & 120.3 & 45.9 & 1154.4 & \textbf{331.6} & 780.6 & \textbf{73.3} \\
    \hline
  \end{tabular}
  \caption{Median computation time, in seconds, required to solve the problems described in Table~\ref{tab:test_problems} using L-shaped with different aggregation schemes. The times required to solve the deterministic equivalents directly are also reported. The fastest result is marked in bold for each problem. The time taken to solve the deterministic equivalents of \texttt{LandS} and \texttt{gbd} are excluded because they did not finish after relatively long computation times.}
  \label{tab:largescale}
\end{table}

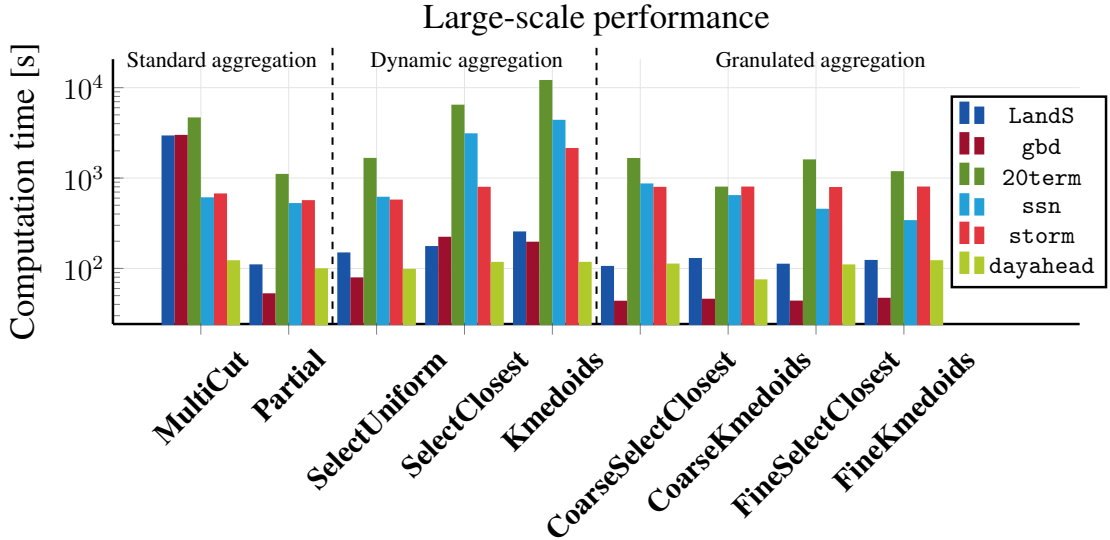
\begin{figure}
  \centering
  \resizebox{0.9\textwidth}{!}{
  \begin{tikzpicture}[]
    \begin{semilogyaxis}[height = {53.269444444444446mm}, legend pos = {north east}, ylabel = {Computation time [s]}, title = {Large-scale performance}, unbounded coords=jump,scaled x ticks = false, xlabel = {}, xlabel style = {font = {\fontsize{14 pt}{14.3 pt}\selectfont}, rotate = 0.0},xmajorgrids = true, xtick align = inside,xticklabel style = {font = {\fontsize{12 pt}{10.4 pt}\selectfont}, rotate = 45.0},x grid style = {color = kth-lightgray,
        line width = 0.25,
        solid},axis x line* = left,x axis line style = {line width = 1,
        solid},scaled y ticks = false,ylabel style = {font = {\fontsize{14 pt}{14.3 pt}\selectfont}, rotate = 0.0},ymajorgrids = true,ytick align = inside,yticklabel style = {font = {\fontsize{12 pt}{10.4 pt}\selectfont}, rotate = 0.0},y grid style = {color = kth-lightgray,
        line width = 0.25,
        solid},axis y line* = left,y axis line style = {line width = 1,
        solid},    xshift = 0.0mm,
      yshift = 0mm,
      ybar = 1pt,
      bar width = 0.15cm,
      xtick = data,
      xmax = extra2,
      symbolic x coords={\textbf{MultiCut},\textbf{Partial},\textbf{SelectUniform},\textbf{SelectClosest},\textbf{Kmedoids},\textbf{CoarseSelectClosest},\textbf{CoarseKmedoids},\textbf{FineSelectClosest},\textbf{FineKmedoids},extra,extra2},
      title style = {font = {\fontsize{14 pt}{18.2 pt}\selectfont}, rotate = 0.0},legend style = {line width = 1, xshift = 8mm, yshift = -4mm,
        solid,font = {\fontsize{10 pt}{10.4 pt}\selectfont}},colorbar style={title=}, ymin = {0}, width = {152.4mm}]
      \addplot+ [color = kth-blue,
      line width = 1,
      solid,mark = none,
      mark size = 2.0,
      mark options = {
        color = black,
        fill = kth-blue,
        line width = 1,
        rotate = 0,
        solid
      }] coordinates {
        (\textbf{MultiCut}, 2867.881116888)
        (\textbf{Partial}, 107.46696166)
        (\textbf{SelectUniform}, 145.626648228)
        (\textbf{SelectClosest}, 171.39198896)
        (\textbf{Kmedoids}, 248.67242511)
        (\textbf{CoarseSelectClosest}, 103.271004108)
        (\textbf{CoarseKmedoids}, 126.342183312)
        (\textbf{FineSelectClosest}, 109.313701678)
        (\textbf{FineKmedoids}, 120.305342864)
      };
      \addlegendentry{\texttt{LandS}}
      \addplot+ [color = kth-red,
      line width = 1,
      solid,mark = none,
      mark size = 2.0,
      mark options = {
        color = black,
        fill = kth-red,
        line width = 1,
        rotate = 0,
        solid
      }] coordinates {
        (\textbf{MultiCut}, 2910.149777178)
        (\textbf{Partial}, 51.412536623)
        (\textbf{SelectUniform}, 77.146026107)
        (\textbf{SelectClosest}, 217.12229874)
        (\textbf{Kmedoids}, 191.462208758)
        (\textbf{CoarseSelectClosest}, 42.558758058)
        (\textbf{CoarseKmedoids}, 44.770351933)
        (\textbf{FineSelectClosest}, 42.684789195)
        (\textbf{FineKmedoids}, 45.881097568)
      };
      \addlegendentry{\texttt{gbd}}
      \addplot+ [color = kth-green,
      line width = 1,
      solid,mark = none,
      mark size = 2.0,
      mark options = {
        color = black,
        fill = kth-green,
        line width = 1,
        rotate = 0,
        solid
      }] coordinates {
        (\textbf{MultiCut}, 4541.182708927)
        (\textbf{Partial}, 1076.90947167)
        (\textbf{SelectUniform}, 1619.205119759)
        (\textbf{SelectClosest}, 6284.675698785)
        (\textbf{Kmedoids}, 11785.137523813)
        (\textbf{CoarseSelectClosest}, 1615.73448623)
        (\textbf{CoarseKmedoids}, 779.245290129)
        (\textbf{FineSelectClosest}, 1558.698238719)
        (\textbf{FineKmedoids}, 1154.382193278)
      };
      \addlegendentry{\texttt{20term}}
      \addplot+ [color = kth-lightblue,
      line width = 1,
      solid,mark = none,
      mark size = 2.0,
      mark options = {
        color = black,
        fill = kth-lightblue,
        line width = 1,
        rotate = 0,
        solid
      }] coordinates {
        (\textbf{MultiCut}, 593.440517846)
        (\textbf{Partial}, 512.214536437)
        (\textbf{SelectUniform}, 601.730427227)
        (\textbf{SelectClosest}, 3028.826667795)
        (\textbf{Kmedoids}, 4273.528127389)
        (\textbf{CoarseSelectClosest}, 844.658442669)
        (\textbf{CoarseKmedoids}, 627.606488302)
        (\textbf{FineSelectClosest}, 443.388912742)
        (\textbf{FineKmedoids}, 331.626749674)
      };
      \addlegendentry{\texttt{ssn}}
      \addplot+ [color = kth-lightred,
      line width = 1,
      solid,mark = none,
      mark size = 2.0,
      mark options = {
        color = black,
        fill = kth-lightred,
        line width = 1,
        rotate = 0,
        solid
      }] coordinates {
        (\textbf{MultiCut}, 654.694296607)
        (\textbf{Partial}, 551.99496934)
        (\textbf{SelectUniform}, 559.949202313)
        (\textbf{SelectClosest}, 775.796054255)
        (\textbf{Kmedoids}, 2082.262125634)
        (\textbf{CoarseSelectClosest}, 774.492389134)
        (\textbf{CoarseKmedoids}, 780.216737453)
        (\textbf{FineSelectClosest}, 770.963001891)
        (\textbf{FineKmedoids}, 780.622261702)
      };
      \addlegendentry{\texttt{storm}}
      \addplot+ [color = kth-lightgreen,
      line width = 1,
      solid,mark = none,
      mark size = 2.0,
      mark options = {
        color = black,
        fill = kth-lightgreen,
        line width = 1,
        rotate = 0,
        solid
      }] coordinates {
        (\textbf{MultiCut}, 119.318433163)
        (\textbf{Partial}, 97.160135059)
        (\textbf{SelectUniform}, 95.734061138)
        (\textbf{SelectClosest}, 114.50394825)
        (\textbf{Kmedoids}, 114.580121021)
        (\textbf{CoarseSelectClosest}, 109.679055157)
        (\textbf{CoarseKmedoids}, 73.320672212)
        (\textbf{FineSelectClosest}, 107.299728785)
        (\textbf{FineKmedoids}, 119.484407286)
      };
      \addlegendentry{\texttt{dayahead}}
    \end{semilogyaxis}
    \node at (1.54,3.7) {\small Standard aggregation};
    \draw[dashed,thick] (3.1,0) -- (3.1,4);
    \node at (5,3.7) {\small Dynamic aggregation};
    \draw[dashed,thick] (6.83,0) -- (6.83,4);
    \node at (10,3.7) {\small Granulated aggregation};
  \end{tikzpicture}
}
  \caption{Median computation time required to solve the problems described in Table~\ref{tab:test_problems} using L-shaped with different aggregation schemes.}
  \label{fig:large_scale}
\end{figure}

\section{Discussion and conclusion}
\label{sec:disc-concl}

\subsection{Discussion}
\label{sec:discussion}

There is no single strategy that outperforms the others for every problem. The optimal parameter configurations are also not the same. This implies that the best aggregation scheme is problem-dependent. However, we can propose some rules-of-thumb based on these experiments. First, among the proposed distance measures in~\ref{sec:distance-measures}, the angular distance appears to be most suited for distance-based aggregation scheme. In general, the granulated strategies show great promise mostly in combination with the k-medoids clustering scheme. The \textbf{Kmedoids} rule itself showed promise in the small-scale experiment, but suffered from scalability issues for the large-scale problems. However, these issues are successfully alleviated by the granulated strategy.

Like prior approaches, our dynamic aggregation schemes are governed by tunable parameters. Our small- and large-scale experiments indicate that the sensitivity to the parameter choice is reduce by pre-granulation. Instead of guessing the optimal aggregation level, one can choose either a coarse or fine pre-granulation and then and apply a dynamic aggregation scheme with aggressive or conservative aggregation settings. The experimental results on the diverse testset indicate that this strategy will on average yield performance improvements.

In all experiments, we observe solid performance results from partial cut aggregation. It consistently outperforms the bare-bone dynamic schemes. We can relate this observation to our worst-case results. Although the dynamic aggregation schemes we propose could theoretically aggregate cuts in a more clever way, they could also theoretically identify more facets than static schemes before converging. This is supported by our worst-case bound on the dynamic aggregation~\eqref{eq:dynaggcomplexity} which in general is expected to be larger than the static worst-case bound~\eqref{eq:aggcomplexity} because of the combinatorial terms. It is however outperformed by the granulated strategies on most problems. There could exist better aggregation levels that would improve the performance of partial cut aggregation. However, parameter tuning is not feasible when solving large-scale problems. The granulated strategies appears more reliable in this regard because they outperform partial cut aggregation on five out of six different problems using general rules. A greater set of problems should be considered to further test this hypothesis, but these initial results are promising.

Our derived worst-case bounds grow astronomically large quickly, but they do not give accurate estimates of average-time complexity. Instead, they allow us to reason about aggregation schemes and suggest rules of thumb. From practical experience, we would not expect the dynamical worst-case bound~\eqref{eq:dynaggcomplexity} to be attained by anything but diabolically constructed problems. An identified facet in some aggregate generally corresponds to many other facets in coarser aggregates. The worst case would therefore occur only if all facets are identified in a very specific order, which is unlikely in the average case. Hence, the combinatorial explosion suggested by the worst-case bound is rarely observed in practice. Future work could involve further theoretical development around the average-time complexity of these algorithms.

\subsection{Conclusion}
\label{sec:conclusion}

In this work, we have presented a novel framework for dynamic cut aggregation in L-shaped algorithms. With our approach, the optimality cuts generated at each iteration can be aggregated into arbitrary partitions which are allowed to vary at each iteration. We have given a worst-case bound for aggregated L-shaped in Theorem~\ref{thm:aggcomplexity} that holds for any static partition scheme $\mathcal{S}$. We have also extended this worst-case result to dynamic aggregation in Theorem~\ref{thm:dynaggcomplexity} and given a convergence proof for L-shaped with dynamic cut aggregation in Theorem~\ref{thm:dynlshapedconvergence}. We have proposed three practical aggregation types, dynamic aggregation, cluster aggregation, and granulated aggregation, and also introduced various decision rules that yield a large set of dynamic aggregation schemes.

The proposed aggregation schemes have been evaluated by solving a diverse set of large-scale stochastic programs, which are distributed over $\num{32}$ worker nodes. Although the best aggregation scheme and parameter configuration are unknown for a given problem, we have shown that large performance gains are attainable with granulated aggregation using generally applicable rules. Our set of proposed aggregation schemes do not encompass every possible partitioning scheme and we aim to explore more strategies in the future. Based on our experimental observations, and our worst-case bound, we suggest designing aggregation schemes that limit the possible combinations of aggregates. In brief, our experimental results are promising and indicate that granulated aggregation in combination with k-medoids clustering can yield significant performance improvements for distributed L-shaped algorithms.

\appendix

\section{Proofs}
\label{sec:proofs}

In this section, we provide the proofs of Theorems~\ref{thm:aggcomplexity}-~\ref{thm:dynaggcomplexity}.

\begin{proof}[Proof of Theorem~\ref{thm:aggcomplexity}]
  In the worst case, a single facet of one of the $A(\mathcal{S})$ aggregates is identified at each iteration, so that all facets are identified before converging. Hence, because $A(\mathcal{S})$ facets are identified in the first iteration, the maximum number of iterations is $1+M-A(\mathcal{S})$, where $M$ is the total number of facets that can be identified in all aggregates. Consider any of the aggregates $\mathcal{S}_a$. In the worst case, $b_s = b$ for every $Q_s(x)$ in that aggregate. If so, it holds that one facet of this aggregate, in every direction $j$, consists of facets from each of its $\abs{\mathcal{S}_a}$ constituents, for a total of $\abs{\mathcal{S}_a}b$ combinations. However, the facet identified in the considered aggregate at the first iteration consists of $\abs{\mathcal{S}_a}$ facets because $\theta_a$ is initially unrestricted in the master problem. In the worst case, any new facet identified in aggregate $\mathcal{S}_a$ includes only one facet that has not been identified before. There are $b-1$ such slopes remaining for each of the constituents, for a total number of $1+\abs{\mathcal{S}_a}(b-1)$ facets in $\mathcal{S}_a$. Moreover, this can occur in all $m$ dual directions of the subproblems. Hence, the maximum number of iterations required to identify all facets in the given aggregate $\mathcal{S}_a$ is given by $\brackets*{1+\abs{\mathcal{S}_a}(b-1)}^m$, and hence, $M = \sum_{a = 1}^{A(\mathcal{S})} \brackets*{1+\abs{\mathcal{S}_a}(b-1)}^m$. In conclusion, the maximum number of iterations of the aggregated L-shaped is in the worst case given by
  \begin{equation*}
    1 + \sum_{a = 1}^{A(\mathcal{S})} \brackets*{1+\abs{\mathcal{S}_a}(b-1)}^m - A(\mathcal{S}).
  \end{equation*}
\end{proof}

\begin{proof}[Proof of Theorem~\ref{thm:dynlshapedconvergence}]
  We assume without loss of generality that every iterate $x_k$ generated during the L-shaped algorithm is second-stage feasible. Otherwise, we can fallback to the standard proof using a finite number of feasibility cuts. Now, for any partitioning scheme that satisfies~\eqref{eq:partition-conditions}, the $A_k$ optimality cut aggregates generated during one iteration will form supports of the second-stage objective because
  \begin{equation*}
    \theta = \sum_{a = 1}^{A}\sum_{s \in \mathcal{S}^k_a} \theta_{s,k} \geq \sum_{a = 1}^{A}\sum_{s \in \mathcal{S}^k_a} \pi_s \lambda_{s,k}^T\parentheses*{h_s-T_sx} = \sum_{s = 1}^{N} \pi_s\lambda_{s,k}^T(h_s-T_sx)
  \end{equation*}
  for some $\parentheses*{\lambda_{1,k},\dots,\lambda_{N,k}} \in \bar{\Lambda}_1\times\dots\times\bar{\Lambda}_N$, which is exactly one of the facets of $Q(x)$. Every iteration a new iterate $x_k$ and $\braces{\theta_{s,k}}_{s = 1}^{N}$ is obtained from solving the master problem. Now, it can hold that
  \begin{equation*}
    \sum_{s \in \mathcal{S}^k_a} \theta_{s,k} < \sum_{s \in \mathcal{S}^k_a} \parentheses*{q_{s,k} - \partial Q_{s,k} x_k}
  \end{equation*}
  for some, or all, of the current iteration aggregates $\mathcal{S}^k_a \in \mathcal{S}^k$. If so, the current set of aggregated cuts in the master do not impose
  \begin{equation*}
    \sum_{s = 1}^{N} \theta_s \geq Q(x).
  \end{equation*}
  Therefore, a new set of second-stage dual multipliers, not already present in the master problem, will be added through aggregated optimality cuts. Because each set of extreme points $\bar{\Lambda}_s$ is finite, this can only occur finitely many times. Therefore, it must eventually hold that
  \begin{equation*}
    \sum_{s \in \mathcal{S}^k_a} \theta_{s,k} \geq \sum_{s \in \mathcal{S}^k_a} \parentheses*{q_{s,k} - \partial Q_{s,k} x_k}
  \end{equation*}
  for all $a = 1,\dots,A_k$ so that
  \begin{equation*}
    \theta_k = \sum_{a = 1}^{A}\sum_{s \in \mathcal{S}^k_a} \theta_{s,k} \geq \sum_{a = 1}^{A}\sum_{s \in \mathcal{S}^k_a} \pi_s \lambda_{s,k}^T\parentheses*{h_s-T_s x_k} = \sum_{s = 1}^N \pi_s \lambda_{s,k}^T\parentheses*{h_s-T_s x_k}.
  \end{equation*}
  Now, since $\theta_k$ is optimal and the $\theta_{s,k}$ are free in~\eqref{eq:dynamicls} except for the cut constraints, it follows that
  \begin{equation*}
    \theta_k = Q(x_k) = \sum_{s = 1}^{N} \pi_s\max_{\lambda_{s,k} \in \bar{\Lambda}_s}{\lambda_{s,k}^T(h_s-T_s x_k)} \leq \sum_{s = 1}^{N} \pi_s\max_{\lambda_{s,k} \in \bar{\Lambda}_s}{\lambda_{s,k}^T(h_s-T_sx)} = Q(x).
  \end{equation*}
  In conclusion, $x_k$ is an optimal solution to~\eqref{eq:sp}.
\end{proof}

\begin{proof}[Proof of Theorem~\ref{thm:dynaggcomplexity}]
  In the worst case, a single facet of one of the $A_k$ aggregates is identified at each iteration $k$, so that all possible combinations of facets are identified before converging. Hence, because $A_0$ facets are identified in the first iteration, the maximum number of iterations is $1+M-A_0$, where $M$ is the total number of facets that can be identified in all possible aggregates. Consider any aggregate $\mathcal{S}^k_a$ at some iteration $k$. We have already shown that the number of facets that can be identified in this aggregate is given by $\brackets*{1+\abs{\mathcal{S}^k_a}(b-1)}^m$. There are no assumed restrictions on the partitioning schemes in $\mathcal{D}$. Therefore, any aggregate of the same size as $\mathcal{S}^k_a$ could be considered in subsequent iterations, each of which share the same number of possible facets that can be identified. If the common size is denoted by $a_L$, this number is given by $\brackets*{1+a_L(b-1)}^m$. The number of aggregates that share the size $a_L$ is given by the number of combinations of $a_L$ out of $N$. Moreover, the size of a given aggregate can vary between $1$ and $N$. Therefore, the total number of facets identifyable in all possible aggregates is given by $\sum_{a_L = 1}^{N} \binom{N}{a_L} \brackets*{1+a_L(b-1)}^m$. When the algorithm has converged it will hold that all possible facets corresponding to some partitioning scheme $\mathcal{S} \in \mathcal{D}$ have been identified. Furthermore, in the worst case, there is only one facet in all other possible partitioning schemes that have not been identified before the final iteration. These facets will not be identified since the algorithm terminates; so, their total must be subtracted from the number of facets we can consider. This number is equal to the total number of possible partitioning schemes minus one due to the scheme active during the final iteration. The number of possible partitioning schemes is given exactly by the Bell number. Therefore, the maximum number of iterations required to converge is in the worst case given by:
  \begin{equation*}
    1 + \sum_{a_L = 1}^{N} \binom{N}{a_L} \brackets*{1+a_L(b-1)}^m - (B_N-1) - A_0 = 2 + \sum_{a_L = 1}^{N} \binom{N}{a_L} \brackets*{1+a_L(b-1)}^m - \sum_{a_L = 1}^{N} \stirling{N}{a_L} - A_0.
  \end{equation*}
\end{proof}

\section{Distance measures}
\label{sec:distance-measures}

Many of the devised heuristics for selecting which cuts to aggregate require a measure of distance between two given optimality cuts. Let $c_s$ denote a generated optimality cut on the form
\begin{equation} \label{eq:optimalitycut}
  \partial Q_s x + \theta_s \geq q_s
\end{equation}
and let $d(c_i,c_j)$ denote some distance measure between two optimality cuts of the form~\eqref{eq:optimalitycut}. We do not devise measures that fulfill all conditions of a metric, but we at least require that $d(c_i,c_j) \geq 0$ and that $d(c_i,c_j) = 0$ whenever $c_i = c_j$. Ideally, we want a measure so that $c_i$ and $c_j$ give similar information about the feasible region in the master problem when $d(c_i,c_j)$ is small. To this end, we borrow ideas from the following survey paper about aggregation techniques in optimization~\cite{Rogers1991} when exploring measures. We stipulate and utilize the following three measures.

\subsection{Absolute distance}
\label{sec:absolute-distance}

First, we introduce the absolute distance between two optimality cuts as:
\begin{equation} \label{eq:absdist}
  d(c_i,c_j) = \frac{\norm{\tilde{c}_i - \tilde{c}_j}}{\max{\parentheses*{\norm{\vphantom{\tilde{c}_j}\tilde{c}_i}, \norm{\tilde{c}_j}}}}
\end{equation}
where
\begin{equation*}
  \tilde{c}_s = \begin{bmatrix}
 \partial Q_s \\
 q_s
\end{bmatrix}
\end{equation*}
The absolute distance has the property that $d(c_i,c_j) = 0$ precisely when $c_i = c_j$. However, it will often place a heavy weight on $q_s$, since $q_s$ directly relates to the second-stage objective and it often holds that $\abs{q_s} \gg \norm{\partial Q_s}$. In many of the introduced selection rules, a cut candidate $c_s$ is often compared to an existing aggregate $c_{\mathcal{S}_a}$ of cuts:
\begin{equation} \label{eq:aggcut}
  \sum_{s \in \mathcal{S}_a}\partial Q_s x + \sum_{s \in \mathcal{S}_a} \theta_s \geq \sum_{s \in \mathcal{S}_a} q_s.
\end{equation}
Due to the summation, the distance between an aggregated cut and a single cut will generally be larger than that between two single cuts. Therefore, we normalize by the number of cuts when calculating the distance, so that
\begin{equation*}
  \tilde{c}_{\mathcal{S}_a} = \frac{1}{\abs{\mathcal{S}_a}}\begin{bmatrix}
 \sum_{s \in \mathcal{S}_a}\partial Q_s \\
 \sum_{s \in \mathcal{S}_a} q_s
\end{bmatrix}.
\end{equation*}

\subsection{Angular distance}
\label{sec:angular-distance}

Next, we introduce the angular distance between two cuts as
\begin{equation} \label{eq:angulardist}
  1 - \frac{\abs{\partial Q_i \cdot \partial Q_j}}{\norm{\vphantom{\partial Q_j}\partial Q_i}\norm{\partial Q_j}}.
\end{equation}
This distance is invariant over aggregation; so, there is no need to rescale. The maximum distance is acquired for perpendicular cuts, which are probably undesired to aggregate. The main drawback is that the distance between parallel cuts is zero.

\subsection{Spatioangular distance}
\label{sec:spat-dist}

Finally, we introduce the spatioangular distance between two cuts as
\begin{equation*}
  1 - \frac{\abs{\partial Q_i \cdot \partial Q_j}}{\norm{\vphantom{\partial Q_j}\partial Q_i}\norm{\partial Q_j}} + \frac{\abs{q_i-q_j}}{\max{\parentheses*{\abs{\vphantom{q_j}q_i},\abs{q_j}}}}.
\end{equation*}
This formulation alleviates the drawback of the angular distance by also measuring the distance between the bias terms $q_i$ and $q_j$. However, it is not as straightforward to decide at what relative tolerance the two cuts should be considered close enough for aggregation. As with the absolute distance, we again keep track of the amount of cuts included in an aggregate and rescale $q_s$ accordingly.

\bibliographystyle{unsrt}
\bibliography{references}

\end{document}